\documentclass[reqno]{amsart}
\usepackage{relsize}

\usepackage{graphicx}
\usepackage{amsfonts, amsmath}
\usepackage{xcolor}
\usepackage{amssymb}
\usepackage{latexsym}
\usepackage{exscale}
\usepackage{comment}

\usepackage{tikz}

\allowdisplaybreaks
\theoremstyle{plain}
\newtheorem{lemma}{Lemma}[section]
\newtheorem{theorem}[lemma]{Theorem}
\newtheorem{corollary}[lemma]{Corollary}

\newtheorem{proposition}[lemma]{Proposition}
\newtheorem{definition}[lemma]{Definition}

\theoremstyle{remark}
\newtheorem{remark}{Remark}

\makeatletter
\newcommand*{\rom}[1]{\expandafter\@slowromancap\romannumeral #1@}
\makeatother

\def\ee{\varepsilon}

\def\pa{\partial}
\def\ee{\end{equation}}
\def\be{\begin{equation}}

\begin{document}

\vskip 0.125in

\title[Finite-time Blowup and stability of blowup]
{Stable singularity formation for the Inviscid Primitive Equations}

\date{\today}
\author[C. Collot]{Charles Collot*}
\address[C. Collot]
{ \thanks{*Corresponding author.}
Laboratoire Analyse, G\'eom\'etrie, et mod\'elisation,
	CNRS and CY Cergy Paris Universit\'e
	 2 rue Adolphe Chauvin \\
	95300 Cergy, France.}
 \email{ccollot@cyu.fr}

\author[S. Ibrahim]{Slim Ibrahim}
\address[S. Ibrahim]
{Department of Mathematics and Statistics  \\
University of Victoria  \\
3800 Finnerty Road, Victoria  \\
 B.C., Canada V8P 5C2.} 
 \email{ibrahims@uvic.ca}

\author[Q. Lin]{Quyuan Lin}
\address[Q. Lin]
{	Department of Mathematics \\
     University of California  \\
	Santa Barbara, CA 93106, USA.} \email{quyuan\_lin@ucsb.edu}

\begin{abstract}
The primitive equations (PEs) model large scale dynamics of the oceans and the atmosphere. While it is by now well-known that the three-dimensional viscous PEs is globally well-posed in Sobolev spaces, and that there are solutions to the inviscid PEs (also called the hydrostatic Euler equations) that develop singularities in finite time, the qualitative description of the blowup still remains undiscovered. In this paper, we provide a full description of two blowup mechanisms, for a reduced PDE that is satisfied by a class of particular solutions to the PEs. In the first one a shock forms, and pressure effects are subleading, but in a critical way: they localize the singularity closer and closer to the boundary near the blow-up time (with a logarithmic in time law). This first mechanism involves a smooth blow-up profile and is stable among smooth enough solutions. In the second one the pressure effects are fully negligible; this dynamics involves a two-parameters family of non-smooth profiles, and is stable only by smoother perturbations. 

\end{abstract}

\maketitle

MSC Subject Classifications: 35B44, 35Q86, 86A10.\\

Keywords: inviscid primitive equations; hydrostatic Euler equations; blow-up; stability

\section{introduction}
We consider the $3D$ inviscid primitive equations (PEs) 
\begin{subequations}\label{PE-system-intro}
\begin{align}
    &u_t + u u_X + vu_Y + w u_Z  +
	p_X - \Omega v= 0, \label{EQ1-1}
	\\
	&v_t + u v_X + vv_Y + w v_Z  + p_Y + \Omega u= 0 , \label{EQ1-2} 
	\\
	&p_Z + T =0 ,   \label{EQ1-3}  
	\\
	&\theta_t + u \theta_X + v \theta_Y + w \theta_Z  = 0, \label{EQ1-4} 
	\\
	&u_X + v_Y+ w_Z =0,   \label{EQ1-5}
\end{align}
\end{subequations}
set in the domain 
\begin{equation*}
    \mathcal D = \mathcal M \times [0,1] = \{(X,Y,Z): (X,Y)\in \mathcal M, 0\leq Z\leq 1\},
\end{equation*}
where $\mathcal M \subseteq \mathbb R^2$ is a smooth bounded domain with real analytic boundary. System \eqref{PE-system-intro} is supplemented with the initial value $(u_0, v_0, \theta_0)$, and satisfies the relevant geophysical boundary conditions (cf. \cite{KTVZ11}):
\begin{equation}\label{PE-system-intro-BC}
\begin{split}
     &w(t,X,Y,Z) = 0 \text{ on } Z=0 \text{ and } Z=1;
     \\
     &\int_0^1 (u,v)(t,X,Y,Z) dZ \cdot \vec{n} = 0 \text{ on } (X,Y)\in \partial \mathcal M,
\end{split}
\end{equation}
where $\vec{n}$ is the outward unit normal to $\partial\mathcal M$.
System \eqref{PE-system-intro} is derived as a formal asymptotic limit of the small aspect ratio (the ratio of the depth or the height to the horizontal length scale) from the Boussinesq system (see \cite{AG01,LT18}). With full viscosity, the global existence of strong solutions for the $3D$ PEs was firstly established in \cite{CT07}, and later in \cite{HK16,K06,KZ07,KZ072}. The above results were extended to the cases with only horizontal viscosity, see  \cite{CLT16,CLT17,CLT17b}. With only vertical viscosity, the ill-posedness in Sobolev spaces is shown in \cite{RE09}. This ill-posedness can be overcome by considering additional linear (Rayleigh-like friction) damping \cite{CLT19}, or Gevrey regularity and some convex conditions on the initial data \cite{GMV20}.

In the absence of viscosity, and for adiabatic systems (i.e. constant temperature that can be set by convention to be zero), 
system \eqref{PE-system-intro} is also called the hydrostatic Euler equations. Such system has a loss of horizontal derivative, making local well-posedness in Sobolev a hard problem for general initial data. Indeed, 
the linear and nonlinear ill-posedness in any Sobolev space have been established in \cite{RE09} and in \cite{HN16}, respectively. 
On the other hand, by assuming either real analyticity
or some special structures (local Rayleigh condition) on the initial data, one is able to establish the local well-posedness, see \cite{BR99,BR03,GILT20,GR99,KMVW14,KTVZ11,MW12}. Moreover, it was proven that smooth solutions to the inviscid PEs, in the absence of rotation, can develop singularities in finite time. (cf. \cite{CINT15,W12}). Recently, it is shown in \cite{ILT20} that the results about ill-posedness and finite-time blowup can be extended to the case with rotation.

The different proofs of the singularity formation rely on the proof of the finite time blowup for the corresponding $2D$ model for which the initial data can be lifted to the full three-dimensional equation, where the well-posedness is applicable. To get to the $2D$ system, the ideas in \cite{CINT15, ILT20} is to observe that if $\Omega =0$ and
 initially $\theta_0 = v_0= 0$, then any smooth enough solution $(u,v,w,\theta)$ to system \eqref{PE-system-intro}, with initial data $(u_0, v_0, \theta_0)$ and boundary condition \eqref{PE-system-intro-BC}, must satisfy $\theta(t,X,Y,Z) \equiv 0$ and $v(t,X,Y,Z) \equiv 0$. Moreover, if initially $u_0$ is independent of the $Y$ variable, then any smooth enough solution remains independent of the $Y$ variable. Therefore, under the assumption that we have a smooth solution and for the following initial data
\begin{equation} \label{2d-condition}
    u_0(X,Y,Z)=u_0(X,Z), \;\; v_0(X,Y,Z)=0, \;\;  \theta_0(X,Y,Z) =0, 
\end{equation}
we obtain the $2D$ inviscid PEs system (also known as the hydrostatic Euler equations):
\begin{subequations}\label{PE-system-intro-2d}
\begin{align}
    &u_t + u\, u_X + w u_Z +p_X  = 0 , \label{EQ2-1}
    \\
    &p_Z =0 ,   \label{EQ2-3}  
    \\
    &u_X + w_Z =0.   \label{EQ2-4}
\end{align}
\end{subequations}
Since system \eqref{PE-system-intro-2d} is independent of $Y$ variable, the horizontal domain $\mathcal M$ needs to be $Y$ independent and translation invariant in $Y$. Hence without loss of generality we may consider system \eqref{PE-system-intro-2d} set on 
\begin{equation*}
    \mathcal D = \big\{(X,Z):  -L\leq X \leq L, 0\leq Z\leq 1\big\},
\end{equation*} 
and the boundary conditions \eqref{PE-system-intro-BC} becomes
\begin{equation}\label{PE-system-intro-2d-BC}
    \begin{split}
        &w(t,X,0) = w(t,X,1) = 0, \\
        &\int_0^1 u(t,-L,Z) dZ = \int_0^1 u(t,L,Z) dZ = 0.
    \end{split}
\end{equation}
From (\ref{EQ2-4}) and (\ref{PE-system-intro-2d-BC}), we know that
\begin{equation} \label{ux-barotropic}
   w(t,X,Z) = -\int_0^Z u_X(t,X,\tilde Z) d\tilde Z, \quad \int_0^1 u_X(t,X,Z) dZ = 0, \text{ and thus, } \int_0^1 u(t,X,Z) dZ = 0.
\end{equation}
One is able to further simplify system \eqref{PE-system-intro-2d}. Differentiating (\ref{EQ2-1}) with respect to $X$, one obtains
\begin{eqnarray}\label{EQx-1} 
u_{Xt} + u\, u_{XX} + u_X^2 + w_X u_Z + wu_{XZ} +p_{XX} = 0 ,  
\end{eqnarray}
Thanks to (\ref{ux-barotropic}), integrating (\ref{EQx-1}) with respect to $Z$ over the interval $[0,1]$, an integration by parts together with (\ref{EQ2-3}),  (\ref{EQ2-4}) and (\ref{PE-system-intro-2d-BC}) enables us to solve for the pressure:
\begin{eqnarray}
&&\hskip-.8in  p_{XX} = -\int_0^1  (u^2)_{XX}   dZ.\label{EQpxx}
\end{eqnarray}
We consider from now on solutions $u$ that are odd in $X$, i.e. $u(X)=-u(-X)$, and introduce the trace of their horizontal derivative on the central line:
\be \label{def:a}
a(t,Z)=-\pa_X u(t,0,Z)
\ee
Differentiating \eqref{EQx-1} with respect to $X$, then injecting \eqref{ux-barotropic} and \eqref{EQpxx}, and taking $X=0$, one obtains the following closed evolution equation for $a$:
\begin{subequations}\label{1D-model-intro}
\begin{align}
    &a_t -a^2 + \Big(\int_0^Z a(t,\widetilde{Z})d\widetilde{Z}\Big) a_Z +2\int_0^1 a^2 dZ =0, \label{1D-model-intro1}
    \\
    &\int_0^1 a(t,Z)\;dZ=0. \label{1D-model-intro2}
\end{align}
\end{subequations}
Note that for solutions $u$ having the form
\begin{equation}\label{special-form}
    u(t,X,Z) = -Xa(t,Z), \text{ and thus, } u_X(t,X,Z) =  -a(t,Z),
\end{equation}
system \eqref{PE-system-intro-2d} and the boundary condition \eqref{PE-system-intro-2d-BC} for $u$ are equivalent to system \eqref{1D-model-intro} for $a$. We emphasize here that the term $2\int_0^1 a^2 dZ$ comes from the pressure term.

\begin{remark}\label{remark:rotation}
 In the presence of rotation i.e. $\Omega\geq 0$, choosing $-\pa_X v_0(0,Z) = \Omega$, it has been shown in \cite{ILT20} that $a$ defined in \eqref{def:a} still satisfies system \eqref{1D-model-intro}.
\end{remark}

In \cite{CINT15}, a family $(\psi_m)_{m>0}$ of initial data has been constructed for which the corresponding solution $a$ to \eqref{1D-model-intro} blows up at time $t=1$ with $a(t,Z)=(1-t)^{-1}\psi_m(Z)$. Lifting this result to a blowup for the original 2D system \eqref{PE-system-intro-2d} is however non-trivial given the lack of well-posedness result in the class of regularity of the profiles $\psi_m$. In addition perturbation of these solutions seems challenging given their rigidity. On the other hand, Wong \cite{W12} has constructed explicit initial data to \eqref{PE-system-intro-2d} that are analytic for which the corresponding solution will exhibit a singularity in finite time making  $\|u(t,\cdot)\|_{L^\infty}+\|p_X(t,\cdot)\|_{L^\infty}+\|u_{X}(t,\cdot)\|_{L^\infty}$ infinite at the blowup time. The purpose of this paper is to provide precise qualitative properties on the singularity formation.


Our main result is the following showing the existence and stability of a blowup solution for \eqref{1D-model-intro} with a smooth profile.

\begin{theorem}[Smooth blowup]\label{theorem:critical}
Consider the profile $\phi(z)=\phi_0(z)=e^{-z}$. Then there exist  $\lambda^{*}_0>0$ and $\delta>0$ such that for all $0<\lambda_0\leq \lambda_0^*$ a constant $\kappa>0$ exists such that, if initially
\be \label{smooth:id:initial}
a_0(Z)=\frac{1}{\lambda_0} \phi\left(\frac{Z}{\nu_0}\right)+\tilde a_0(Z) , \ \ 0\leq Z \leq 1,
\ee
with
\be \label{smooth:id:initialcond}
\frac{2}{3\log (\lambda_0^{-1})}\leq \nu_0\leq\frac{3}{2\log (\lambda_0^{-1})},\quad
\mbox{and}\quad\left\|\tilde a_{0}\right\|_{C^2 ([0,1])} \leq \kappa,
\ee
then there exists $T>0$ and $C>0$ such that the solution $a$ to \eqref{1D-model-intro} with initial data $a(t=0)=a_0$ blows up at time $T>0$ according to
\be \label{smooth:1}
a(t,Z)=\frac{1}{(T-t)} \phi\left(\frac{Z}{\nu(t)}\right)+\tilde a(t,Z) \qquad \mbox{with} \quad \nu(t)=\frac{1}{|\log(T-t)|},
\ee
where for all $t\in [0,T)$:
\be \label{smooth:2}
\|\tilde a(t,\cdot)\|_{L^{\infty}([0,1])}\leq C (T-t)^{-1} |\log (T-t)|^{-\delta}
\ee

\end{theorem}

\begin{remark} 
\hfill \break
\begin{itemize}
    \item Note that the general solution $u$ (other than \eqref{special-form}) to \eqref{PE-system-intro} might not exist up to time $T$. If it does, then the divergence $\| (\pa_X u)|_{X=0}(t)\|_{L^\infty([0,1])}\rightarrow \infty$ as $t\uparrow T$ implied by \eqref{def:a}, \eqref{smooth:1} and \eqref{smooth:2} signals the formation of a shock for $u$ along the horizontal $X$-direction at the point $(X,Z)=(0,0)$.
    
   \item The pressure term is of lower order compared with the other terms as $t\uparrow T$, but not negligible. Indeed, the modulation equations for the scaling parameters $(\lambda,\nu)$ are:
   \be \label{eq:modulationintro}
   \left\{\begin{array}{l l} \lambda_t=-1+C_0 \nu +h.o.t., \\  \nu_t=-C_0 \frac{\nu^2}{\lambda} +h.o.t., \end{array} \right., \qquad C_0= 2\int_0^\infty \phi_0^2(z)dz=1.
   \ee
   Here $h.o.t.$ mean higher order terms. The pressure is thus responsible for the behaviour \eqref{smooth:1} for $\nu$ corresponding to a self-similarity of the second kind.
   
\item Equation \eqref{1D-model-intro1} shares similarities with the viscous Prandtl equation on the axis \cite{EE}:
\begin{align}\label{1DPrandtl}
\left\{\begin{array}{ll}
\xi_t-\xi_{yy}-\xi^2+\left(\int_0^y \xi \right)\xi_Y=0, \quad Y>0\\
\xi(t,0)=0, \ \ \xi(0,Y)=\xi_0(Y).
\end{array}
 \right.
\end{align}
In \cite{CGIM18}, a blow-up dynamics is found, for which the viscosity is negligible, where the singularity forms on a large spatial scale $y\approx \nu(t)=(T-t)^{-1/2}$, see also \cite{CGM,CGM2,KVW}. Note that \eqref{1D-model-intro}  without pressure is \eqref{1DPrandtl} without viscosity. Interestingly, the blow-up dynamics of Theorem \ref{theorem:critical} and that of \cite{CGIM18} are genuinely different, due to the absence of a Dirichlet boundary condition for \eqref{1D-model-intro} allowing for a blow-up at the boundary, and to the absence of pressure and confinement $z\in [0,1]$ in \eqref{1DPrandtl} allowing for the transverse spatial scale $\nu$ to grow to infinity so to maintain the divergence free condition.
\item Another instance of a blow-up dynamics with a logarithmic correction for the scale due to subleading but non-negligible effects happens for the semilinear heat equation \cite{MZ}.

\end{itemize}
\end{remark}

The existence and uniqueness of analytic solutions of system \eqref{PE-system-intro} in the domain $\mathcal D$ with boundary condition \eqref{PE-system-intro-BC} is established in \cite{KTVZ11}. By virtue of this, Theorem \ref{theorem:critical}, and Remark \ref{remark:rotation}, we have the following corollary regarding the existence of an explicit singular solution of system \eqref{PE-system-intro} satisfying the boundary condition \eqref{PE-system-intro-BC}.
\begin{corollary}
 Consider the profile $\phi(z)=e^{-z},$ and suppose that the initial condition $(u_0,v_0,\theta_0)$ of system \eqref{PE-system-intro} with boundary condition \eqref{PE-system-intro-BC} satisfies
\begin{equation*}
\begin{split}
    &u_0(X,Y,Z) =  -Xa_0(Z) = -X\left( \frac{1}{\lambda_0} \phi\left(\frac{Z}{\nu_0}\right)+\tilde a_0(Z) \right),
    \\
    &v_0(X,Y,Z) = -\Omega X, \qquad \theta_0(X,Y,Z)=0, 
\end{split}
\end{equation*}
where $a_0$ is defined as in Theorem \ref{theorem:critical}. Then the unique analytic solution blows up at time $T>0$ stated in Theorem \ref{theorem:critical}.
\end{corollary}

In the second result, we study a blowup regime where the pressure can be neglected around the blowup time. We show that this can happen through a two-parameters $(\beta,\tilde \nu)$ family of non-smooth profiles.

\begin{theorem}[Non-smooth blowup]\label{nonsmooth}
For any $\beta>0$, there exists a non-smooth profile function $ \phi_\beta\in\mathcal C^\frac1{\beta+1}([0,\infty))$ with $\phi_\beta>0$, $\phi_\beta(0)=1$ and $\phi_\beta$ decreasing on $[0,\infty)$, and a constant $\delta>0$ such that the following holds true.

For any $\tilde \nu_0^*>0$, there exists $\lambda^{*}_0>0$ and for all $0<\lambda_0\leq \lambda^*_0$, a constant $\kappa>0$, such that if initially
\be \label{nonsmooth:id:initial}
a_0( Z)=\frac{1}{\lambda_0} \phi_\beta \left(\frac{Z}{\nu_0}\right)+\tilde a_0(Z), \ \ 0\leq Z \leq 1,
\ee
with
\be \label{nonsmooth:id:initialcond}
0<\lambda_0\leq \lambda^*_0,\quad 0<\tilde \nu_0=\frac{\nu_0}{\lambda_0^\beta}\leq \tilde \nu^*_0,\quad
\mbox{and}\quad\left\|\tilde a_{0}\right\|_{C^{1}([0,1])} \leq \kappa,
\ee
then there exist $T>0$, $\tilde \nu_\infty>0$, $C>0$ such that the solution $a$ to \eqref{1D-model-intro} with $a(t=0)=a_0$ blows up at time $T>0$ with:
\be \label{nonsmooth:1}
a(t, Z)=\frac{1}{T-t} \phi_\beta \left(\frac{Z}{\nu(t)}\right)+\tilde a(t, Z), \qquad \mbox{with}\quad \nu(t)=\tilde \nu_\infty( T-t)^\beta,
\ee
where for all $t\in [0,T)$:
\be \label{nonsmooth:2}
\|\tilde a (t)\|_{L^{\infty}([0,1])} \leq C(T-t)^{-1+\delta}.
\ee
\end{theorem}

\begin{remark}
\hfill \break
\begin{itemize}
   \item The profile $\phi_\beta$ is almost explicit, see Proposition \ref{proposition:profile}. The parameter $\tilde \nu$ comes from the fact that for a fixed $\beta>0$ the full family of blow-up profile is $(\phi_{\beta,\tilde \nu})_{\tilde \nu>0}$ where $\phi_{\beta,\tilde \nu}(z)=\phi_{\beta }(z/\tilde \nu)$.   
\item The fact that it is not smooth is crucial: the soft\footnote{Soft in the sense that it still allows for local in time existence of a solution.} singularity of the initial data plays a crucial role in the mechanism underlying the (worst) singularity at time $T$. This is a feature of the hyperbolic nature of \eqref{1D-model-intro}. Similar soft singularities playing a role during finite time blowup are found for the Burgers equation \cite{CGM}, the nonlinear wave equation \cite{Kr}, and the incompressible Euler equations \cite{ElGindi21,EGM19}. However, the construction of stable smooth blowup seems in general to be more challenging. For instance, it still remains open for the 3D Euler equations.

    \item Compared to Theorem \ref{theorem:critical}, the modulation equations for the scale parameters $(\lambda,\nu)$ are:
   \be \label{eq:modulationintro2}
   \left\{\begin{array}{l l} \lambda_t=-1+C_\beta \nu +h.o.t., \\ \nu_t=-\beta \frac{\nu}{\lambda}-C_\beta(\beta+1) \frac{\nu^2}{\lambda}+h.o.t., \end{array} \right. \qquad C_\beta= 2\int_0^\infty \phi_\beta^2(z)dz,
   \ee
   for which the pressure effects (the second terms in the right-hand side of \eqref{eq:modulationintro2}) are negligible. Comparing \eqref{eq:modulationintro2} and \eqref{eq:modulationintro}, we see that the $\beta=0$ case \eqref{eq:modulationintro} is \emph{critical} for the role of the pressure which drives the blow-up scale $\nu$ to $0$ and slightly slows down the blow-up speed (through the $C_0 \nu$ term in \eqref{eq:modulationintro}). Interestingly, this critical regime corresponds in fact to a stable blow-up scenario among smooth solutions.
\end{itemize}
\end{remark}

\section{Strategy of the proof and organization of the paper}\label{sec:organization}

To explain the idea of our proof, let us first aim at studying blow-up solution under the assumption that 
\be \label{hp}
\mbox{Hypothesis:}\quad \mbox{Assume } \int_0^1a^2\;dZ \mbox{ is of lower order and the confinement }Z\in [0,1] \mbox{ is irrelevant}.
\ee
In such a case, the dynamics is given to the leading order by:
\begin{equation} \label{eq:a:simplified}
    \hskip-.8in a_t -a^2 + \Big(\int_0^Z a(t,\widetilde{Z})d\widetilde{Z}\Big) a_Z  =0,\quad 0\leq Z<\infty.
\end{equation}
Note that if $a(t,Z)$ is a solution of system \eqref{eq:a:simplified} then, for any $\mu, \lambda > 0$, $\frac1\lambda a(\frac t\lambda,\frac Z\mu)$ is also a solution. We look for a self-similar solution to \eqref{eq:a:simplified} of the form
\begin{eqnarray}\label{ansatz}
a(t,Z)=\frac1{T-t}\phi_\beta \Big(\frac Z{(T-t)^\beta}\Big).
\end{eqnarray}
Denote by 
$$
z=\frac Z{(T-t)^\beta},
$$
then the profile $\phi_\beta$ has to satisfy
$$
(\beta z + \partial_z^{-1} \phi_\beta) \frac{d\phi_\beta}{dz} - \phi^2_\beta + \phi_\beta = 0.
$$
Note that for the Ansatz \eqref{ansatz}, $a_t$, $a^2$, and $\big(\int_0^Z a(t,\tilde Z)d\tilde Z\big) a_Z$ are of size $(T-t)^{-2}$, and that $\int_0^1a^2 \;dZ$ is of size $(T - t)^{-2+\beta}$. Hence, Hypothesis \eqref{hp} would be satisfied only in the {\it subcritical case} $\beta>0$. However, when $\beta=0$, note that, for any $\nu > 0$, the function
$$
a(t,Z)= \frac1{T-t}\phi_0\Big(\frac Z\nu\Big)
$$
is a solution to \eqref{eq:a:simplified}. The term $\int_0^1 a^2\;dZ$ is of size $\nu (T-t)^{-2}$, and $a_t$, $a^2$, and $\big(\int_0^Z a(t,\tilde Z)d\tilde Z\big) a_Z$ are of size $(T-t)^{-2}$. Hence, the Hypothesis \eqref{hp} is satisfied when the effect of the term $\int_0^1a^2\;dZ$ is to drive the parameter $\nu\to 0$ as $t\to T$.\\

\noindent The paper is organized as follows. Section \ref{sec:profile} is devoted to the construction in Proposition \ref{proposition:profile} of our family of profiles $(\phi_\beta)_{\beta>0}$. The proof is similar to Proposition 3.2. in \cite{CGIM18}. In section \ref{sec:nonsmooth} we construct the stable blowup in the non-smooth case, where the decay is exponential in time in similarity variables. The proof uses a {\it bootstrap} argument by defining a basin of attraction of the blowup in Definition \ref{definition:trap-subcritical}. Our proof goes in three steps. First, we find and solve the modulation equations in Lemma \ref{lemma:modulation} by imposing suitable zero boundary conditions on the perturbation. Next, we estimate the remainder $\tilde a$ in an interior region in Lemma \ref{lemma:interior}. We use a Lyapunov functional with a spatial weight that penalizes the nonlocal terms inspired from \cite{CGIM18}. Its decay in time is due to the repulsivity of the transport field and potential terms generated by the $\phi_\beta$ profile, which results in a spectral-gap like coercivity, in analogy with \cite{MRRS}. Last, we estimate the solution in an exterior region away from the singularity in Lemma \ref{lemma:main} using the maximum principle. Following the same steps of section \ref{sec:nonsmooth}, in section \ref{sec:smooth} we give the proof of Theorem \ref{theorem:critical}. The key difference is due to the modulation equations in Lemma \ref{lemma:smoothmodulation}, which need to go up to quadratic order in the $\nu$ equation. The decay rate of the remainder becomes algebraic in this case. A slightly more refined analysis is needed.

\section{Profiles}\label{sec:profile}

For $\beta\geq0$, consider the profile equation
\begin{eqnarray}\label{profile}
(\beta z + \partial_z^{-1} \phi) \frac{d}{dz}\phi - \phi^2 + \phi = 0,\quad z\in[0,\infty),
\end{eqnarray}
where $\partial_z^{-1} \phi(z)=\int_0^z \phi(\tilde z)d\tilde z$. Note that equation \eqref{profile} has a scaling invariance: if $\phi$ solves \eqref{profile} then so does $\phi_{\tilde \nu}$ defined as $\phi_{\tilde \nu}(z)=\phi(z/\tilde \nu)$ for any $\tilde \nu>0$. The classification of bounded solutions to \eqref{profile} is given by:

\begin{proposition}\label{proposition:profile}
\hfill \break
\begin{itemize}
    \item When $\beta>0$, equation
 \eqref{profile} has a solution $\phi_\beta\in C([0,\infty))\cap C^\infty ((0,\infty))$ satisfying for a constant $d_\beta>0$

\begin{eqnarray}\label{asympt z to 0}
\phi_{\beta}(z)=
1-z^{\frac1{\beta+1}}+z^\frac2{\beta+1}+o(z^\frac2{\beta+1}),\quad\text{as}\quad z\to0,
\end{eqnarray}
and 
\begin{eqnarray}\label{asympt z to infty}
\phi_{\beta }(z)= d_\beta z^{-\frac1\beta}+o(z^{-\frac1\beta}),\quad\text{as}\quad z\to\infty.
\end{eqnarray}

\item When $\beta=0$, equation \eqref{profile} has the explicit solution:
$$
\phi_0(z)=e^{-z}.
$$

\item For all $\beta \geq 0$, if $\phi \in C([0,\infty))\cap C^1 ((0,\infty))$ solves \eqref{profile}, then either $\phi$ is unbounded on $[0,\infty)$, or $\phi$ is constant equal to $0$ or $1$, or there exists $\tilde \nu>0$ such that $\phi(z)=\phi_\beta (z/\tilde \nu)$ for all $z\geq0$.

\end{itemize} 
\end{proposition}

\begin{remark}

Note that $\int_{0}^{\infty} \phi^{2}$ is finite if and only if $0\leq \beta<2$ while $\int_{0}^{\infty} (\phi')^{2}$ is finite if and only if $0<\beta<1$. Note that there might exist solutions to \eqref{profile} on $[0,\infty)$ that are unbounded, but they are however not relevant for our purpose. 

\end{remark}

\begin{proof}

We first consider $\beta>0$, and assume that $\phi=\phi_\beta \in C([0,\infty))\cap C^1 ((0,\infty))$ solves \eqref{profile} and is bounded. \\

\noindent \textbf{Step 1.} \emph{Preliminary properties}. We claim that either $\phi=0$ or $\phi=1$, or $0<\phi(z)<1$ for all $z>0$. 
To prove it, we let $\psi(z)=\beta z+\int_0^z \phi(\tilde z)d\tilde z$, and claim that $\psi(z_0)\neq 0$ for some $z_0>0$. Indeed, if not then $\pa_z \psi=0$, so that $\phi=-\beta$. But this is not a solution to \eqref{profile}. Then, defining $X=(\phi,\psi)$ with $\frac{d}{dz} X=(\pa_z \phi,\pa_z \psi)$ and $ X(z_0)=X_0$, we have that $X$ solves the following ODE whenever $\psi\neq 0$:
\begin{equation} \label{id:ode profile}
\left\{ \begin{array}{l l} \frac{d}{dz} X=(\frac{\phi^2-\phi}{\psi},\beta+\phi),  \\
X_0=(\phi(z_0),\psi(z_0)) .\end{array} \right.
\end{equation}
Note that \eqref{id:ode profile} has the explicit solutions $(0,C+\beta z)$ and $(1,C+(\beta+1) z)$ for some $C\in \mathbb R$. If $X$ is one of such solutions, since $\psi(z)=\beta z+\int_0^z \phi(\tilde z)d\tilde z$, then $C=0$ and these are the $\phi=0$ or $\phi=1$ solutions to \eqref{profile}, respectively. 

Next, We introduce the sets
\begin{align*}
Z_1^\pm =\{0<\phi<1, \ \pm \psi>0\}, \quad Z_2^\pm =\{\phi>1, \ \pm \psi>0\}, \quad Z_3^\pm =\{\phi<0, \ \pm \psi>0\}.
\end{align*}
Assuming that $\phi\neq 0,1$, there exists $i\in \{1,2,3\}$ and $\iota \in \{\pm\}$ such that $X_0\in Z_i^\iota$. We claim that all cases except $X_0\in Z_1^+$ lead to a contradiction, which will prove the claim of Step 1.

If $X_0\in Z_1^-$ or $X\in Z_2^-$, then we notice by a direct check that both sets are invariant by the backward flow of \eqref{id:ode profile}, so that $X(z)\in Z_1^-$ or $X(z)\in Z_2^-$ for all $0<z<z_0$. This implies $\pa_z \psi>0$ for $z<z_0$ and hence $\psi(z)\leq \psi(z_0)<0$ which contradicts $\psi(z)=\beta z+\int_0^z \phi \rightarrow 0$ as $z\to 0$.

If $X_0\in Z_2^+$, then this set is invariant by the forward flow of \eqref{id:ode profile}, so that $X(z)\in Z_2^+$ for all $z>z_0$. This implies $\pa_z \phi(z)>0$ and hence $1<\phi(z_0)\leq \phi(z)$ for all $z\geq z_0$. Moreover, by the boundness of $\phi$, we know $\phi(z)\leq C$ for some $C>1.$ 
We then get using \eqref{id:ode profile} that for $z\geq z_0$,
$\pa_z \phi(z) \geq \frac{\phi^2(z_0)-\phi(z_0)}{\psi(z_0)+ (\beta + C )(z-z_0)} \gtrsim \frac1{1+z}.$ Integrating this inequality gives $\phi(z) \rightarrow +\infty$ as $z\rightarrow \infty$, a contradiction since $\phi$ is bounded.

If $X_0\in Z_3^+$, then we first claim that $X(z)\in Z_3^+$ for all $0<z\leq z_0$. By contradiciton, if not there would exist $0<z_1<z_0$ such that $X(z)\in Z_3^+$ for $z_1<z\leq z_0$ and $X(z_1)\notin Z_3^+$. In this case, $\phi'(z)>0$ for all $z\in (z_1,z_0]$ so that $\phi(z)\leq \phi(z_0)<0$. As $X(z_1)\notin Z_3^+$ and $\phi(z_1)<0$ we must have $\psi(z_1)=0$. Since $\psi>0$ on $Z_3^+$, and $\psi(z)=\beta z+\int_0^z \phi$ with $\phi$ is continuous, we obtain that $0<\psi(z)\leq C (z-z_1)$ on $(z_1,z_0]$ for some constant $C>0$. Hence for $z\in (z_1,z_0]$ we have using \eqref{id:ode profile} that $\pa_z \phi \geq \frac{\phi^2(z_0)-\phi(z_0)}{C(z-z_1)}$, which integrated with $z$ implies $\phi(z) \rightarrow -\infty$ as $z\downarrow z_1$, a contradiction.

If $X_0\in Z_3^-$, we get similarly as for the $Z_3^+$ case that $X\in Z_3^-$ on $[z_0,\infty)$. Then $\pa_z \phi<0$ for $z\geq z_0$, and $\phi$ diverges to $-\infty$ by an argument similar to that used for $Z_2^+$, giving again a contradiction.\\

\noindent \textbf{Step 2.} \emph{Exact formula}. We now assume thanks to Step 1 that $0<\phi(z)<1$ for all $z>0$ and hence $\phi'<0$ and $\psi>0$ on $(0,\infty)$. With similar arguments as in Step 1 above, we obtain that $\phi(z)\rightarrow 0$ as $z\to \infty$, and that $\phi(z)\rightarrow 1$ as $z\to 0$. In particular there exists $z_1>0$ such that 
$$
\phi(z_1)=\frac 12.
$$
We now perform the change of variables on $[0,+\infty)$ by defining $\xi$ such that :
\begin{equation} \label{id:changevar}
\frac{d\xi}{dz}=\frac{\xi}{\beta z+\int_0^z \phi(\tilde z)d\tilde z}, \quad \xi(z_1)=1, \quad \mbox{and}\quad \ H(\xi):=\phi(z),
\end{equation}
so that the equation \eqref{profile} becomes $H-H^2+\xi\partial_\xi H=0$ whose solution, since $H(1)=1/2$, is
\begin{equation} \label{id:inter2}
H=(1+\xi)^{-1}.
\end{equation}
Now, differentiating \eqref{id:changevar} gives
$$
\frac{d^2 z}{d\xi^2} =\frac{d}{d\xi}\left[\frac{\beta z+\int_0^z \phi(\tilde z)d\tilde z}{\xi} \right]= -\frac 1 \xi \frac{dz}{d\xi}+\frac{dz}{d\xi}\frac{\beta+\phi(z)}\xi
=\frac{dz}{d\xi}\left[\frac\beta\xi-\frac 1{1+\xi}\right].
$$
After integration this yields, for $C$ an integration constant:
\begin{equation} \label{id:inter}
\frac{dz}{d\xi}= \frac{C\xi^\beta}{\xi+1}.
\end{equation}
Using \eqref{id:changevar} and $0<\beta z+\int_0^z \phi(\tilde z)d\tilde z<(1+\beta)z$ we obtain that $\lim_{z\downarrow 0} \log \xi(z)= -\int_0^{z_1} [\beta z+\int_0^z \phi(\tilde z)d\tilde z]^{-1}dz=-\infty$ so that $\lim_{z\downarrow 0}\xi(z)=0$. Moreover, we recall that $z(\xi=1)=z_1$. These two considerations and \eqref{id:inter} give:
\begin{equation} \label{id:inter3}
z(\xi)=C \int_0^\xi \frac{\tilde \xi^\beta}{\tilde \xi+1} d\tilde \xi, \qquad C=z_1 C_\beta, \quad C_\beta = \left(\int_0^1 \frac{\tilde \xi^\beta}{\tilde \xi+1} d\tilde \xi\right)^{-1}.
\end{equation}
For any $z_1>0$, the identities \eqref{id:changevar}, \eqref{id:inter2} and \eqref{id:inter3} provide solutions to \eqref{profile}, which are equal up to the scaling transformation $\phi(\cdot) \mapsto \phi(\cdot/\tilde \nu)$. We also proved these are the only possible bounded solutions. To get the asymptotics of $\phi$ at $z\to0$, and $z\to\infty$, we integrate \eqref{id:inter3}:
$$
z(\xi)= C \int_0^\xi\frac{u^\beta}{1+u}\;du=
C \sum_{j\geq0}\frac{(-1)^j}{\beta+j+1}\xi^{\beta+j+1}
,\; \text{for}\; 0\leq \xi<1. 
$$
This implies that

$$
\xi=\big(\frac{\beta+1} C  \big)^\frac1{\beta+1}z^\frac1{\beta+1}+O(z^{1+\frac1{\beta+1}}),
$$
and consequently, using \eqref{id:changevar} and \eqref{id:inter2}:
$$
\phi(z)=\frac1{1+z^\frac1{\beta+1}\left(\big(\frac{\beta+1}C\big)^\frac1{\beta+1}+O(z)\right)}
$$
which implies \eqref{asympt z to 0} upon choosing $z_1=\frac{\beta+1}{C_\beta}$. Similarly, one can derive the asymptotic as $\xi\;(\text{equivalently}\;z)\to\infty$. \\

\noindent \textbf{Step 3.} \emph{The case $\beta=0$}: The case $\beta=0$ can be solved explicitly by separating variables in the differential equation \eqref{profile}. Indeed, letting $\psi=\partial_z^{-1}\phi$, we get
$$
\psi\psi''=\psi'(\psi'-1)
$$
giving the one-parameter family of solutions $\psi=\tilde \nu(e^{z/\tilde \nu}-1)$ for $\tilde \nu>0$, i.e. $\phi=e^{- z/\tilde \nu}$. This finishes the proof of the proposition.
\end{proof}

For the sake of simplicity, we shall first give the proof of Theorem \ref{nonsmooth} for the case of non-smooth blowup.

\section{Non-smooth blowup}\label{sec:nonsmooth}
In this section, we focus on the case $\beta>0.$ 
\subsection{Derivation of the rescaled model in self-similar coordinates}
Consider the following rescaling for $\nu$ and $\lambda$ two positive $C^1$ functions of time:
\begin{equation} \label{id:self-similarvariables}
    z = \frac{Z}{\nu(t)}, \quad \frac{ds}{dt} = \frac{1}{\lambda(t)}, \quad s(0)=s_0.
\end{equation}
The following computations are done for all $\beta\geq 0$, wih $\phi=\phi_\beta$ (we drop the $\beta$ subscript to ease notation) given by Proposition \ref{proposition:profile}. We write the solution $a(t,Z)$ of system \eqref{1D-model-intro} as
\begin{equation}\label{decomp}
    a(t,Z) = \frac{1}{\lambda(s(t))}\Big(\phi\Big(\frac Z{\nu(s(t))}\Big) + \varepsilon\Big(s(t),\frac Z{\nu(s(t))}\Big)\Big)= \frac{1}{\lambda(s)}\Big(\phi (z) + \varepsilon (s,z)\Big).
\end{equation}
From the explicit computations:
\begin{equation*}
\begin{split}
    &a_t = \frac{1}{\lambda^2} \Big( -\frac{\lambda_s}{\lambda}\phi - \phi' \frac{\nu_s}{\nu} z - \frac{\lambda_s}{\lambda} \varepsilon + \varepsilon_s - \varepsilon_z \frac{\nu_s}{\nu} z\Big),
    \\
    &a^2 = \frac{1}{\lambda^2} \Big( \phi^2 + \varepsilon^2 + 2\phi \varepsilon \Big),
    \\
    &\Big(\int_0^Z a(t,\widetilde{Z})d\widetilde{Z}\Big) a_Z = \frac{1}{\lambda^2} \Big( \partial_z^{-1} \phi \phi' + \partial_z^{-1} \phi \varepsilon_z + \partial_z^{-1} \varepsilon \phi' + \partial_z^{-1} \varepsilon \varepsilon_z \Big),
    \\
    &2 \int_0^1 a^2(Z) dZ = \frac{2\nu}{\lambda^2} \int_0^{\frac{1}{\nu}} ( \phi+ \varepsilon)^2(z) dz,
\end{split}
\end{equation*}
thanks to \eqref{profile}, system \eqref{1D-model-intro} gives
\begin{subequations}\label{equation:epsilon}
\begin{align}
    \begin{split}\label{equation:epsilon-1}
        &\varepsilon_s - \frac{\lambda_s}{\lambda}\varepsilon - \frac{\nu_s}{\nu} z \varepsilon_z - 2\phi \varepsilon + \partial_z^{-1} \phi \varepsilon_z + \partial_z^{-1} \varepsilon \phi'  - \varepsilon^2 + \partial_z^{-1} \varepsilon \varepsilon_z 
        \\
        &\qquad \qquad = (\frac{\lambda_s}{\lambda} +1)\phi + (\beta + \frac{\nu_s}{\nu})z\phi' - 2\nu \int_0^{\frac{1}{\nu}} (\phi+ \varepsilon)^2(z) dz,
    \end{split}
    \\
    &\int_0^{\frac1\nu} (\phi + \varepsilon)(z) dz = 0 \label{equation:epsilon-bc}.
\end{align}
\end{subequations}

The modulation parameters $\lambda$ and $\nu$ are determined by imposing the following vanishing for the expansion of $\epsilon$, an orthogonality-like condition:
\be \label{orthogonality}
\left\{ \begin{array}{l l}  \varepsilon(s,z=0)=0 ,\\
 \pa_z \varepsilon(s,z)=O(z^{\frac{2}{\beta+1}-1}) \qquad \mbox{as }z\downarrow 0. \end{array} \right.
\ee
In order to have the boundary condition $\varepsilon(z=0)=0$, \eqref{equation:epsilon-1} gives the first modulation equation
\begin{equation}\label{equation:modulation-1}
    \frac{\lambda_s}{\lambda} +1 = 2\nu \int_0^{\frac{1}{\nu}} (\phi + \varepsilon)^2 (z) dz.
\end{equation}
By taking the derivative of \eqref{equation:epsilon-1} with respect to $z$, one obtains
\begin{equation}\label{equation:epsilon_z}
\begin{split}
     &\varepsilon_{zs} - (\frac{\lambda_s}{\lambda} + \frac{\nu_s}{\nu} + \phi)  \varepsilon_{z}      + (\partial_z^{-1} \phi - \frac{\nu_s}{\nu} z ) \varepsilon_{zz} - \phi' \varepsilon + \partial_z^{-1} \varepsilon \phi''
     - \varepsilon \varepsilon_z + \partial_z^{-1} \varepsilon \varepsilon_{zz}
    \\
    = & (\frac{\lambda_s}{\lambda} +1 + \beta + \frac{\nu_s}{\nu})\phi' + (\beta + \frac{\nu_s}{\nu})z\phi''.
\end{split}
\end{equation}
Since near zero
$
    \phi(z) = 1 - z^{\frac{1}{\beta+1}} + z^{\frac{2}{\beta+1}} + o(z^{\frac{2}{\beta+1}}),
$
then near zero one has
\begin{equation*}
    \phi' = \frac{-1}{\beta+1} z^{\frac{-\beta}{\beta+1}} + \frac{2}{\beta+1} z^{\frac{1-\beta}{\beta+1}} + o(z^{\frac{1-\beta}{\beta+1}}), \qquad z\phi'' = \frac{\beta}{(\beta+1)^2} z^{\frac{-\beta}{\beta+1}} + \frac{2(1-\beta)}{(\beta+1)^2} z^{\frac{1-\beta}{\beta+1}} + o(z^{\frac{1-\beta}{\beta+1}}) .
\end{equation*}
The second boundary condition in \eqref{orthogonality} then gives the second modulation equation
\begin{equation}\label{equation:modulation-2}
    \frac{\nu_s}{\nu} = -\beta - (\beta+1)(1+\frac{\lambda_s}{\lambda}) = -\beta - (\beta+1) 2\nu \int_0^{\frac{1}{\nu}} (\phi + \varepsilon)^2 (z) dz.
\end{equation}
Thus, we can rewrite \eqref{equation:epsilon_z} as 
\begin{equation}\label{equation:epsilon_z-2}
\begin{split}
     &\varepsilon_{zs} - (\frac{\lambda_s}{\lambda} + \frac{\nu_s}{\nu} + \phi)  \varepsilon_{z}      + (\partial_z^{-1} \phi - \frac{\nu_s}{\nu} z ) \varepsilon_{zz} - \phi' \varepsilon + \partial_z^{-1} \varepsilon \phi''
    - \varepsilon \varepsilon_z + \partial_z^{-1} \varepsilon \varepsilon_{zz}
    \\
    = & - \Big(\beta \phi' + (\beta+1) z\phi''\Big) 2\nu \int_0^{\frac{1}{\nu}} (\phi + \varepsilon)^2 (z) dz.
\end{split}
\end{equation}

\subsection{Bootstrap Argument}
We now fix once for all $\beta >0$ and its corresponding profile $\phi=\phi_\beta$ given by Proposition \ref{proposition:profile}, and pick constants 
\be \label{id:defeta-defdelta}
0<\eta<\min(\beta,1),  \  \  \delta=2\min(\beta,1)-\eta.
\ee
Our proof of the main result goes using a bootstrap argument. We start by setting up its framework. The notation $C(a_1,a_2,\cdot\cdot\cdot,a_n)$ stands for a generic constant that depends on its arguments $a_1,\cdot\cdot,a_n$ and that may change from line to line.
Consider a solution $a$ of \eqref{1D-model-intro} written in the self-similar coordinates $(s,z)$ given by \eqref{id:self-similarvariables}, and assume it is decomposed in the form \eqref{decomp}. Therefore, $\varepsilon$ and $\varepsilon_z$ satisfy \eqref{equation:epsilon} and \eqref{equation:epsilon_z-2}, respectively. We control $\varepsilon$ with the following quantities
\be \label{id:defmathcalErough}
\mathcal E_1^2(s) = \int_0^{z^*}w(z)(\partial_z\varepsilon(s,z))^2\;dz , \quad \ \mathcal E_2(s) =\sup_{z^*\leq z\leq\frac1{\nu(s)}}|\varepsilon| .
\ee
Here $z^*\geq 1$ will be fixed later on, and the weight function is chosen to be
\be \label{id:defw}
w(z)=z^{\alpha} e^{-Kz}
\ee
with $K$ large enough satisfying \eqref{condition:K} and \eqref{condition:K-2}, below, and $\alpha$ being defined by
\begin{equation}
    \alpha = \frac{|1-\beta|-2+\frac\eta2}{\beta+1} =
    \begin{cases}
     \frac{-\beta-1+\frac{\eta}2}{\beta+1} =-1+\frac{\eta}{2(\beta+1)} &\text{when $0<\beta\leq 1$,}\\
     \frac{\beta-3+\frac{\eta}2}{\beta+1}=-1+\frac{4(\beta-1)+\eta}{2(\beta+1)} =1-\frac4{\beta+1}+\frac{\eta}{2(\beta+1)}  &\text{when $\beta > 1$.}
     \end{cases}
\end{equation}
In particular, notice that $-1< \alpha < 1$ is true for any $\beta>0$ and $\eta$ satisfying \eqref{id:defeta-defdelta}.

\begin{definition}[Initial closeness]\label{definition:initial-subcritical}
Let $\lambda_0^*,\tilde \nu_0^*>0$. We say that $a_0$ is initially close to the blow-up profile if there exist $0<\lambda_0\leq \lambda_0^*$ and $\nu_0>0$ such that 
\begin{itemize}
\item[(i)] \emph{The initial values of the modulation parameters satisfy (note that the first equation fixes the value of $s_0$)}
\begin{eqnarray}
\label{bd:parametersini}
\lambda_0 = e^{-s_0}, \ \ \tilde \nu_0= \nu_0 e^{\beta s_0} \leq \tilde \nu_0^*. 
\end{eqnarray} 

\item[(ii)] \emph{The initial perturbation $\varepsilon(s_0,z)=\varepsilon_0\in C([0,\frac1{\nu_0}])\cap C^{1}((0,\frac1{\nu_0}])$, given by the decomposition \eqref{decomp}, satisfies the integral condition \eqref{equation:epsilon-bc}, the boundary condition \eqref{orthogonality}, as well as the two conditions \eqref{equation:modulation-1} and \eqref{equation:modulation-2}}.

\item[(iii)] \emph{For some small number $\gamma > 0$ satisfying conditions \eqref{condition:delta*} and \eqref{condition:delta*onemore} below, $\varepsilon_0$ satisfies}

\begin{eqnarray}\label{bd:eini}
\mathcal E_1^2(s_0) = \int_0^{z^*}w(z)(\partial_z\varepsilon_0(z))^2\;dz  < \gamma e^{-\delta s_0}, \ \ \mathcal E_2^2(s_0) =\sup_{z^*\leq z\leq\frac1{\nu_0}}|\varepsilon_0|^2 < \frac1{16}e^{-\delta s_0}.
\end{eqnarray} 

\end{itemize}

\end{definition}
Our task is to show that solutions that are initially close to the blow-up profile in the sense of Definition \ref{definition:initial-subcritical} will stay close to this blow-up profile up to modulation. The proximity at later times is defined as follows.

\begin{definition}[Trapped solutions] \label{definition:trap-subcritical}

We say that a solution $a(s,z)$ is trapped on $[s_0,s_1]$ with $s_0<s_1<\infty$, if it satisfies the properties of Definition \ref{definition:initial-subcritical} at time $s_0$ and if for some $\tilde K > 1$ and for all $s\in [s_0,s_1]$, $a(s,z)$ can be decomposed as in \eqref{decomp} with:
\begin{itemize}
\item[(i)] \emph{Values of the modulation parameters:}
\begin{eqnarray}\label{bd:parameterstrap}
\frac 1{\tilde K} e^{-s}< \lambda < \tilde K e^{-s}, \ \  \frac{ \tilde \nu_0}{\tilde K} e^{-\beta s}<\nu < \tilde \nu_0 \tilde Ke^{-\beta s}.
\end{eqnarray} 
\item[(ii)] \emph{Decay in time of the remainder in the self-similar variables:}

\begin{eqnarray}\label{bd:etrap}
\mathcal E_1^2(s) = \int_0^{z^*}w\varepsilon_z(s)^2\;dz  < \tilde K^2e^{-\delta s}, \ \ 
\mathcal E_2^2(s) = \sup_{z^*\leq z\leq\frac1{\nu(s)}}|\varepsilon(s)|^2 < \tilde K^2e^{-\delta s}.
\end{eqnarray} 
\end{itemize}

\end{definition}

The proof of the Theorem \ref{nonsmooth} relies on the following bootstrap proposition.

\begin{proposition} \label{pr:bootstrap}

For any $\beta>0$, $0<\eta<\min (\beta,1)$, and $\delta = 2\min(\beta,1)-\eta$, there exist universal constants $\tilde K,K,z^*\geq 1$ and $\gamma>0$ such that the following holds true. For any $\tilde \nu_0^*>0$, there exists $s_0^*$ large enough such that for all $s_0\geq s_0^*$, any solution of \eqref{1D-model-intro} which is initially close to the blow-up profile in the sense of Definition \ref{definition:initial-subcritical} is trapped on $[s_0,+\infty)$ in the sense of Definition \ref{definition:trap-subcritical}.

\end{proposition}

A standard continuity argument implies that for $s_0$ large enough, any solution which is initially close to the blow-up profile in the sense of Definition \ref{definition:initial-subcritical} is trapped in the sense of Definition \ref{definition:trap-subcritical} on some interval $[s_0,s_1]$ with $s_1>s_0$. Letting $s^*>s_0$ be the supremum of times $s_1> s_0$ such that the solution is trapped on $[s_0,s_1]$, the purpose now is to show that $s^*=+\infty$. The strategy is to study the trapped regime via several lemmas and show that the solutions cannot escape from the open set defined by Definition \ref{definition:trap-subcritical}.

Note that the constant  $s_0^*$ (defined in Proposition \ref{pr:bootstrap}) will be adjusted during the proof: we will always be able to conclude the proof of the various lemmas by choosing $s_0^*$ large enough. By time-shift invariance, we can always assume the original initial time to be $t=0$. First, let us derive \textit{a priori} $L^\infty$ and $L^2$ bounds of the remainder of trapped solutions.

\begin{lemma}\label{LinftyL2 bound}
Given a solution $a$ of \eqref{1D-model-intro} that is trapped on a the interval $[s_0,s_1]$ in the sense of Definition \ref{definition:trap-subcritical}, for all $s_0\leq s\leq s_1$, we have for $C^*(\beta,\eta,K,z^*)= C \frac{e^{Kz^*}(z^*)^{-\alpha}}{(1-\alpha)K}$,
\be \label{bd:epsilon2}
\int_0^{\frac1\nu}\varepsilon^2\;dz  < C^*\tilde K^2e^{-\delta s} \frac1\nu
\ee
and 
\be \label{bd:epsilon}
\sup_{0\leq z\leq\frac1{\nu(s)}}|\varepsilon|^2 <  C^*\tilde K^2e^{-\delta s}. 
\ee

\end{lemma}

\begin{proof}

\textbf{Step 1.} \emph{A preliminary estimate}. We claim that there exists $C>0$ such that for all $\alpha \in (-1,1)$ and $K\geq 1$, for $w$ given by \eqref{id:defw}:
\be \label{bd:wintegrated}
\int_0^z \frac{1}{w(\tilde z)}d\tilde z\leq \frac{C}{(1-\alpha)Kw(z)} \qquad \forall z>0.
\ee
Indeed, recalling $w(z)=z^{\alpha}e^{-Kz}$, by a rescaling argument and then a change of variables:
\be \label{bd:winterdecomp}
\sup_{z>0} w(z)  \int_0^z \frac{1}{w(\tilde z)}d\tilde z =\sup_{z>0} w(\frac zK)  \int_0^{\frac{z}{K}} \frac{1}{w(\tilde z)}d\tilde z=\frac{1}{K} \sup_{z>0} z^{\alpha}e^{-z} \int_0^{z} \tilde z^{-\alpha}e^{\tilde z}d\tilde z.
\ee
If $\alpha\leq 0$ then for all $z>0$:
\be \label{bd:winter1}
z^{\alpha}e^{-z} \int_0^{z} \tilde z^{-\alpha}e^{\tilde z}d\tilde z \leq e^{-z} \int_0^{z} e^{\tilde z}d\tilde z \leq 1.
\ee
If  $\alpha > 0$ then an integration by parts gives
\begin{equation*}
    z^{\alpha}e^{-z} \int_0^{z} \tilde z^{-\alpha}e^{\tilde z}d\tilde z =\frac z{1-\alpha}-\frac{z^\alpha e^{-z}}{1-\alpha}\int_0^z \tilde z^{1-\alpha}e^{\tilde z}\;d\tilde z
\end{equation*} 
Now, since $z \geq {\tilde z}$, we obtain
\be \label{bd:winter3}
z^{\alpha}e^{-z} \int_0^{z} \tilde z^{-\alpha}e^{\tilde z}d\tilde z \leq 
\frac z{1-\alpha} - \frac{e^{-z}}{1-\alpha} \int_0^z \tilde z e^{\tilde z} d\tilde z =
\frac{e^{-z}}{1-\alpha}(e^z-1)\leq\frac1{1-\alpha}.
\ee
Injecting \eqref{bd:winter1} and \eqref{bd:winter3} in \eqref{bd:winterdecomp} shows \eqref{bd:wintegrated}.\\

\noindent \textbf{Step 2.} \emph{Proof of the Lemma}. For all $z\in [0,z^*]$, using the boundary condition $\varepsilon(s,z=0)=0$, Cauchy-Schwarz inequality, and \eqref{bd:wintegrated}, from \eqref{bd:etrap} one obtains that
\begin{align} \label{bd:epsilon:interiorzone}
|\varepsilon(s,z)|^2 &\leq (\int_0^z w\varepsilon_z^2 d\tilde z) (\int_0^z \frac1w d\tilde z)  \leq  \mathcal E_1^2(s)  (\int_0^{z^*} \frac1w d\tilde z)\leq \frac{C}{(1-\alpha)Kw(z^*)} \mathcal E_1^2 \leq C^* \mathcal E_1^2\\
\nonumber  &\leq C^*\tilde K^2 e^{-\delta s}.
\end{align}
This, together with \eqref{bd:etrap}, implies \eqref{bd:epsilon}. Estimate \eqref{bd:epsilon2} is an immediate consequence of \eqref{bd:epsilon} as
$$
\int_0^{\frac1\nu}\varepsilon^2\;dz \leq \frac1{\nu}\sup_{0\leq z\leq z^*}|\varepsilon|^2 \leq C^*\tilde K^2e^{-\delta s}\frac1\nu.
$$
\end{proof}

We will need the following technical estimate in this section.
\begin{lemma}\label{lemma:weight-estimate}
  Suppose that $\beta > 0$, $K\geq 2$ and $z^* \geq 1$. Denote by 
  \begin{equation*}
      A:=\int_0^{z^*} \phi'(z)^2 w(z) ( \int_0^{z} \frac{1}{w}(\widetilde{z}) d\widetilde{z}) dz  
      +  \int_0^{z^*} \phi''(z)^2 w(z) \Big(\int_0^z (\int_0^{\widetilde{z}} \frac{1}{w}(\xi)d\xi )^{\frac{1}{2}} d\widetilde{z} \Big)^2 dz:=A_1+A_2.
  \end{equation*}
For any $\alpha \in (-1,1)$, if $w$ is given by \eqref{id:defw} then
  \begin{equation} \label{bd:nonlocalintermediateestimate}
      A \leq \frac{C(\beta)}{1-\alpha} \begin{cases}
     K^{-1} &\text{when $0<\beta<1$;}\\
     K^{-1} \ln K &\text{when $\beta =1$;}\\
     K^{-\frac2{\beta+1}}&\text{when $\beta > 1$.}
      \end{cases}
  \end{equation}
    
\end{lemma}

\begin{proof}
We recall that for $i=1,2$:
\be \label{id:interexprphi}
|\pa_z^i \phi (z) |\lesssim z^{-\frac{\beta}{\beta+1}+1-i} \mbox{ for }z\leq1, \quad |\pa_z^i \phi(z)| \lesssim z^{-\frac{\beta+1}{\beta}+1-i} \mbox{ for }z\geq1.
\ee
We first consider $A_1$ and decompose:
   $$
   A_1=A^1_1+A^2_1+A_1^3, \quad  A_1=\int_0^{\frac 1K} \phi'(z)^2 w(z) ( \int_0^{z} \frac{1}{w}(\widetilde{z}) d\widetilde{z}) dz, \quad A_1^2=\int_{\frac{1}{K}}^{1}..., \quad A_1^3=\int_{1}^{z^*}... \; .
   $$
 For $0<z\leq K^{-1}$ we have $w(z)\approx z^{\alpha}$ so that using \eqref{id:interexprphi}:
 $$
 A_1^1 \lesssim \int_0^{K^{-1}} z^{-\frac{2\beta}{\beta+1}}\frac{z}{1-\alpha} dz \lesssim C(\beta)\frac{K^{\frac{-2}{\beta+1}}}{1-\alpha}.
 $$
 For $K^{-1}\leq z \leq 1$ we use \eqref{bd:wintegrated} and \eqref{id:interexprphi} so that: 
\be \label{nonlocal:bd:inter2}
 A_1^2\lesssim \int_{K^{-1}}^1 z^{-\frac{2\beta}{\beta+1}} \frac{1}{K(1-\alpha)} dz \lesssim \frac{C(\beta)}{1-\alpha} \left\{ \begin{array}{l l} K^{-1} \qquad \mbox{if } \beta <1, \\ K^{-1}\log K \qquad \mbox{if } \beta =1, \\ K^{\frac{-2}{\beta+1}}  \qquad  \mbox{if } \beta >1. \end{array} \right.
 \ee
  For $z \geq 1$ we use \eqref{bd:wintegrated} and \eqref{id:interexprphi} so that: 
\be \label{nonlocal:bd:inter3}
 A_1^3 \lesssim \int_{1}^{z^*} z^{-\frac{2(\beta+1)}{\beta}} \frac{1}{K} dz \lesssim \frac{C(\beta)}{1-\alpha} K^{-1}.
\ee
 Summing the three inequalities above shows:
\be \label{nonlocal:bd:inter}
 A_1 \lesssim \frac{C(\beta)}{1-\alpha} \left\{ \begin{array}{l l} K^{-1} \qquad \mbox{if } \beta <1, \\ K^{-1}\log K \qquad \mbox{if } \beta =1, \\ K^{\frac{-2}{\beta+1}}  \qquad  \mbox{if } \beta >1. \end{array} \right.
 \ee
 We turn to $A_2$. For $z>0$ we use the Cauchy inequality and that $\tilde z \mapsto \int_0^{\widetilde{z}} \frac{1}{w}(\xi)d\xi $ is increasing to get
\begin{equation*}
\begin{split}
    & w(z)\Big(\int_0^z (\int_0^{\widetilde{z}} \frac{1}{w}(\xi)d\xi )^{\frac{1}{2}} d\widetilde{z} \Big)^2 \leq w(z) \Big(\int_0^z 1 dz\Big) \Big(\int_0^z \int_0^{\widetilde{z}} \frac{1}{w}(\xi)d\xi  d\widetilde{z} \Big)
    \\
    \leq & w(z) z (\int_0^z \int_0^{z} \frac{1}{w}(\xi)d\xi  d\widetilde{z} )  \leq  z^2 w(z) \int_0^z \frac{1}{w(\tilde z)} d \tilde z.
\end{split}
\end{equation*}
Notice using \eqref{id:interexprphi} that $z\phi''$ has the same asymptotic behavior as $\phi'$ near $z=0$ and $z=\infty.$ Therefore, by repeating a similar calculations, one obtains the same estimate \eqref{nonlocal:bd:inter} for the term $A_2$ than for the $A_1$ term. These two estimates show \eqref{bd:nonlocalintermediateestimate}.

\end{proof}

In the sequel, we reintegrate over time the modulation equations and the various energy and pointwise  estimates, to show that the various upper bounds describing the bootstrap cannot be saturated. Proposition \ref{pr:bootstrap} follows immediately from the following three lemmas.

\begin{lemma}[Modulation Equations] \label{lemma:modulation}
For any choice of constants $\tilde K,K,z^*\geq 1$, $ \gamma,\tilde\nu_0^*>0$ and $0<\eta<\min (\beta,1)$, there exists a large self-similar time $s_0^*$ such that for any $s_0\geq s_0^*$, for any solution which is trapped on $[s_0,s_1]$, we have for $s\in [s_0,s_1]$:
\begin{eqnarray}\label{modulationequations}
\left| \frac{\lambda_s}{\lambda}+1\right|\leq C e^{-\frac{\delta}{2}s}, \qquad \left| \frac{\nu_s}{\nu}+\beta \right|\leq C e^{-\frac{\delta}{2}s},
\end{eqnarray}
for $C>0$ independent of the bootstrap constants, and
\begin{eqnarray}\label{bd:boostrap improved parameters}
 \frac 12 e^{-s}\leq \lambda \leq \frac 32 e^{-s}, \qquad \frac{\tilde \nu_0}2 e^{-\beta s}\leq \nu \leq \frac 32 \tilde \nu_0 e^{-\beta s},
\end{eqnarray}
Moreover, if $s_1=\infty$ then there exists some constants $\tilde \lambda_\infty,\tilde \nu_\infty>0$ such that
\begin{eqnarray}\label{bd:boostrap improved parameters2}
\lambda=\tilde \lambda_\infty \big(1+ O(e^{-\frac\delta2 s}) \big)  e^{-s} ,\quad
\nu=\tilde \nu_\infty \big(1+ O(e^{-\frac\delta2 s}) \big) e^{-\beta s }.
\end{eqnarray}
\end{lemma}

\begin{proof}

\textbf{Step 1.} \emph{A preliminary estimate}. We claim that for $s_0$ large enough for all $s\in [s_0,s_1]$:
\be \label{bd:source}
\nu \int_0^{\frac1\nu} (\phi+\varepsilon)^2dz\leq e^{-(\frac{\delta}{2}+\frac{\eta}{4})s}.
\ee
To prove it, first we use the decay \eqref{asympt z to infty} of $\phi(z)$ as $z\to\infty$, and then inject \eqref{bd:parameterstrap} to obtain 
\begin{align} \nonumber
& \nu \int_0^{\frac{1}{\nu}} \phi^2 (z) dz\lesssim C(\beta) \left\{ \begin{array}{l l} \nu \qquad \quad \mbox{if }\beta<2, \\ \nu |\log \nu| \qquad \mbox{if }\beta=2, \\ \nu^{\frac{2}{\beta}} \qquad \quad \mbox{if }\beta>2, \end{array} \right. \leq C(\beta,\tilde \nu_0,\tilde K)  \left\{ \begin{array}{l l} e^{-\beta s} \qquad \mbox{if }\beta<2, \\s e^{- \beta s} \qquad \mbox{if }\beta=2, \\ e^{-2s} \qquad \mbox{if }\beta>2 ,\end{array} \right. \\
& \label{modulation:bdinter1} \leq C(\beta,\tilde \nu_0,\tilde K)  s e^{-( \frac \delta 2+\frac{\eta}{2}) s} \ \leq \ \frac14e^{-( \frac \delta 2+\frac{\eta}{4}) s}, 
\end{align}
where we used $\frac\delta2 + \frac\eta2= \min(\beta,1)$ for the third inequality, and then took $s_0$ large enough for the last inequality. Second, using \eqref{bd:epsilon2} we obtains
\be \label{modulation:bdinter2}
 \nu \int_0^{\frac1\nu}\varepsilon^2\;dz  <  C^*\tilde K^2e^{-\delta s} \leq \frac14e^{-(\frac \delta 2+\frac{\eta}{4}) s},
\ee
where we used $\eta<\delta$ and took $s_0$ large enough for the last inequalities. Combining \eqref{modulation:bdinter1}, \eqref{modulation:bdinter2} and the inequality $(x+y)^2\leq 2(x^2+y^2)$ shows the estimate \eqref{bd:source} we claimed.\\

\noindent \textbf{Step 2.} \emph{Computing $\lambda$}. We rewrite \eqref{equation:modulation-1} as:
\be \label{modulation:id:inter1}
\lambda_s+\lambda = 2\lambda \nu \int_0^{\frac{1}{\nu}} (\phi + \varepsilon)^2 (z) dz.
\ee
Injecting \eqref{bd:source} in \eqref{modulation:id:inter1} shows the first inequality in \eqref{modulationequations}. Using \eqref{bd:parameterstrap} and \eqref{bd:source} gives $\lambda \nu \int_0^{\frac{1}{\nu}} (\phi + \varepsilon)^2 \leq \tilde K e^{-(1+\frac{\delta}{2}+\frac \eta 4)s}\leq  e^{-(1+\frac{\delta}{2})s}$ for $s_0$ large enough, implying:
\be \label{modulation:id:inter2}
\frac{d}{ds}(e^s\lambda)=O( e^{-\frac{\delta}{2}s}) \qquad \mbox{hence} \qquad  \lambda(s)=\big(e^{s_0}\lambda_0 + \int_{s_0}^s O( e^{-\frac\delta2 \tilde s}) d\tilde s\big) e^{-s} .
\ee
Choosing $s_0$ large enough so that $-\frac12\leq\int_{s_0}^s O( e^{-\frac\delta2 \tilde s}) d\tilde s\leq\frac12$ and injecting \eqref{bd:parametersini} in \eqref{modulation:id:inter2} shows
$$
\frac12e^{-s}\leq\lambda(s)\leq \frac32e^{-s}.
$$
In case $s_1=\infty$, injecting \eqref{bd:parametersini} in \eqref{modulation:id:inter2} and rewriting this identity gives:
 \begin{equation*}
     \begin{split}
         \lambda(s)&=\big(1+ \int_{s_0}^\infty O( e^{-\frac\delta2 \tilde s}) d\tilde s  - \int_{s}^\infty O( e^{-\frac\delta2 \tilde s}) d\tilde s\big) e^{-s}
         \\
         &:= \big(\tilde \lambda_\infty - \int_{s}^\infty O( e^{-\frac\delta2 \tilde s}) d\tilde s\big) e^{-s} =\tilde \lambda_\infty \big(1+ O(e^{-\frac\delta2 s}) \big)  e^{-s}.
     \end{split}
 \end{equation*}
The two inequalities above show \eqref{bd:boostrap improved parameters} and \eqref{bd:boostrap improved parameters2} for $\lambda$.\\

\noindent \textbf{Step 3.} \emph{Computing $\nu$}. We rewrite \eqref{equation:modulation-2} as:
 \begin{equation*}
  \nu_s + \beta \nu= - (\beta+1) 2\nu^2 \int_0^{\frac{1}{\nu}} (\phi + \varepsilon)^2 (z) dz .
\end{equation*}
Reasoning exactly as in Step 2, thanks to \eqref{bd:source} and \eqref{bd:parameterstrap}  one obtains the second inequality in \eqref{modulationequations} and:
\begin{equation*}
    \frac{d}{ds}(e^{\beta s}\nu) = O(e^{-\frac\delta2 s}), \quad \mbox{hence} \quad  \nu(s)=\big(e^{\beta s_0}\nu_0 + \int_{s_0}^s O( e^{-\frac\delta2 \tilde s}) d\tilde s\big) e^{-\beta s} .
\end{equation*}
Choosing $s_0$ large enough so that $-\frac12\leq \tilde \nu_0^{-1}\int_{s_0}^s O( e^{-\frac\delta2 \tilde s}) d\tilde s\leq\frac12$ and using \eqref{bd:parametersini} we get:
$$
\frac{\tilde \nu_0}2e^{-\beta s}\leq\nu(s)\leq \frac32 \tilde \nu_0 e^{-\beta s}.
$$
If $s_1=\infty$, thanks to \eqref{bd:parametersini}, one can rewrite
\begin{eqnarray*}
 \nu(s)=\big(\tilde \nu_0 + \int_{s_0}^\infty O( e^{-\frac\delta2 \tilde s}) d\tilde s-  \int_{s}^\infty O( e^{-\frac\delta2 \tilde s}) d\tilde s\big) e^{-\beta s} := \tilde \nu_\infty \big(1+ O(e^{-\frac\delta2 s}) \big) e^{-\beta s},   
\end{eqnarray*}
The two inequalities above show \eqref{bd:boostrap improved parameters} and \eqref{bd:boostrap improved parameters2} for $\nu$.

\end{proof}

\begin{lemma}[Interior Estimate]\label{lemma:interior}

For any $0<\eta<\min (\beta,1)$ and $z^*\geq 1$, there exist $K^*\geq 1$ such that for all $K\geq K^*$ the following holds true. For any constants $\tilde K\geq 1$ and $ \gamma,\tilde\nu_0^*>0$, there exists a large self-similar time $s_0^*$ such that for any $s_0\geq s_0^*$, for any solution which is trapped on $[s_0,s_1]$, we have for $s\in [s_0,s_1]$:
\begin{eqnarray}\label{bd:int e1}
\mathcal E_1^2(s)  \leq 2\gamma e^{-\delta s}.
\end{eqnarray}

\end{lemma}

\begin{proof}

We recall $w(z) = z^{\alpha} e^{-Kz}$ and $\alpha = \frac{|1-\beta|-2+\frac\eta2}{\beta+1}$. Multiplying \eqref{equation:epsilon_z-2} by $w\varepsilon_z$ and integrating over $[0,z^*]$, one obtains that
\begin{equation} \label{interior:id:expression2}
\begin{split}
    &\frac{1}{2}\frac{d}{ds} \int_0^{z^*}  w \varepsilon_z ^2 dz  - \int_0^{z^*} ( \frac{\lambda_s}{\lambda} + \frac{\nu_s}{\nu}  + \phi +\varepsilon)   w \varepsilon_z^2 dz 
     + \int_0^{z^*} (\partial_z^{-1} \phi - \frac{\nu_s}{\nu} z + \partial_z^{-1} \varepsilon) \varepsilon_{zz}  w \varepsilon_z dz
     \\
     &- \int_0^{z^*} \phi' \varepsilon  w \varepsilon_z dz + 
     \int_0^{z^*} \partial_z^{-1} \varepsilon   w \varepsilon_z \phi'' dz \\
     = &-2\nu \Big(\int_0^{\frac{1}{\nu}} (\phi + \varepsilon)^2 (z) dz\Big) \int_0^{z^*}  \Big(\beta \phi' + (\beta+1) z\phi''\Big)  w \varepsilon_z dz.
\end{split}
\end{equation}
Now we estimate all terms in \eqref{interior:id:expression2}.\\

\noindent \underline{Potential and transport terms}. Integrating by parts yields
\begin{equation}  \label{interior:id:expression}
    \begin{split}
        & - \int_0^{z^*} ( \frac{\lambda_s}{\lambda} + \frac{\nu_s}{\nu}  + \phi +\varepsilon)   w \varepsilon_z^2 dz +\int_0^{z^*} (\partial_z^{-1} \phi - \frac{\nu_s}{\nu} z + \partial_z^{-1} \varepsilon) \varepsilon_{zz}  w \varepsilon_z dz
        \\
        = & \Big(\int_0^{z^*} (\phi+\varepsilon)(\tilde{z})d\tilde{z}- \frac{\nu_s}{\nu}z^* \Big) \frac{1}{2}w(z^*) \varepsilon_z^2(z^*) \\
        &+\int_0^{z^*} \left(- \frac{3}{2} \phi-\frac 32 \varepsilon-\frac{\lambda_s}{\lambda}-\frac{1}{2}\frac{\nu_s}{\nu}-\frac{1}{2}\left(\pa_z^{-1}\phi -\frac{\nu_s}{\nu}z+\pa_z^{-1}\varepsilon \right)\frac{w_z}{w}\right) w\varepsilon_z^2dz 
    \end{split}
\end{equation}
To deal with the boundary term, we compute that the transport field is outgoing at $z^*$, i.e.:
\begin{equation}\label{particle-moving}
   \int_0^{z^*} (\phi+\varepsilon)(\tilde{z})d\tilde{z}- \frac{\nu_s}{\nu}z^*\geq \left(\beta - \sup\limits_{0\leq z\leq \frac1{\nu(s)}} |\varepsilon| \right)z^*\geq \left(\beta -\sqrt{C^*}\tilde K e^{-\frac\delta2 s_0}\right)z^*\geq 0, 
\end{equation}
where we used \eqref{equation:modulation-2}, \eqref{modulationequations} and \eqref{bd:epsilon}, and took $s_0$ sufficiently large. This implies:
\be \label{interior:bd:inter2}
 \Big(\int_0^{z^*} (\phi+\varepsilon)(\tilde{z})d\tilde{z}- \frac{\nu_s}{\nu}z^* \Big) \frac{1}{2}w(z^*) \varepsilon_z^2(z^*)\geq 0.
\ee
Now since $w=z^\alpha e^{-Kz}$, we know that $\frac{w_z}w = \alpha z^{-1} - K = \frac{|1-\beta|-2+\frac\eta2}{(\beta+1)z}  - K$. Using this, \eqref{modulationequations} and \eqref{bd:epsilon} we get the first identity (where the constant involved in the $O()$ depends on $\tilde K$, $\alpha$ and $z^*$):
\begin{equation}\label{spectral-gap}
\begin{split}
    &-2\Big(\frac{\lambda_s}{\lambda} +\frac{\nu_s}{2\nu}  + \frac{3}{2}\phi +\frac 32 \varepsilon+ \frac{1}{2}( \partial_z^{-1} \phi - \frac{\nu_s}{\nu} z +\pa_z^{-1}\varepsilon)\frac{w_z}{w} \Big)
    \\
    = &2 +\beta - 3\phi +O(e^{-\frac{\delta}{2}s})- \Big( \partial_z^{-1} \phi + \beta z + O(e^{-\frac{\delta}{2}s}z)  \Big) \Big(\alpha z^{-1}  - K\Big).
\end{split}
\end{equation}
We now distinguish between the cases $-1<\alpha\leq 0$ and $0<\alpha<1$. When $-1<\alpha\leq 0$ using the fact that $\partial_z^{-1}\phi \geq \phi z$ due to $\phi'<0$, and then $\phi\leq 1$, we get
\begin{align*}
   \eqref{spectral-gap} &\geq 2 + \beta - (3\phi +\alpha \phi) - \alpha \beta +O(e^{-\frac{\delta}{2}s}) \  \geq \ 2+\beta -3 - \alpha - \alpha \beta +O(e^{-\frac{\delta}{2}s}) \\
    &\qquad =\beta-1-(\beta+1)(\frac{|1-\beta|-2+\frac\eta2}{\beta+1})+O(e^{-\frac{\delta}{2}s}) \ \geq \ \delta + \frac\eta2 +O(e^{-\frac{\delta}{2}s}) .
\end{align*}
where we used \eqref{id:defeta-defdelta} for the last inequality. When $0<\alpha<1$ notice that then $\beta>1$, and by using the fact that $\phi\leq 1$ and that $\partial_z^{-1}\phi \leq  z$, one obtains
\begin{equation*}
\begin{split}
    \eqref{spectral-gap} &\geq 2 + \beta -3\phi -\alpha(\beta+1)+O(e^{-\frac{\delta}{2}s})
 \geq -1+\beta -(\beta+1)(\frac{|1-\beta|-2+\frac\eta2}{\beta+1})+O(e^{-\frac{\delta}{2}s})
    \\
    & \geq 2-\frac{\eta}{2} +O(e^{-\frac{\delta}{2}s})= \delta + \frac\eta2 +O(e^{-\frac{\delta}{2}s}).
\end{split}
\end{equation*}
In summary, for any $\beta>0$, we get the repulsivity estimate:
\begin{equation}
 -2\Big(\frac{\lambda_s}{\lambda} +\frac{\nu_s}{2\nu}  + \frac{3}{2}\phi +\frac 32 \varepsilon+ \frac{1}{2}( \partial_z^{-1} \phi - \frac{\nu_s}{\nu} z +\pa_z^{-1}\varepsilon)\frac{w_z}{w} \Big) \geq  \delta + \frac\eta2+O(e^{-\frac{\delta}{2}s}),
\end{equation}
and hence deduce the spectral gap like coercivity estimate:
\be \label{interior:bd:inter1}
2\int_0^{z^*} \left(- \frac{3}{2} \phi-\frac 32 \varepsilon-\frac{\lambda_s}{\lambda}-\frac{1}{2}\frac{\nu_s}{\nu}-\frac{1}{2}\left(\pa_z^{-1}\phi -\frac{\nu_s}{\nu}z+\pa_z^{-1}\varepsilon \right)\frac{w_z}{w}\right) w\varepsilon_z^2dz   \geq \left( \delta + \frac\eta2+O(e^{-\frac{\delta}{2}s})\right)\mathcal E_1(s).
\ee
Injecting \eqref{interior:bd:inter2} and \eqref{interior:bd:inter1} in \eqref{interior:id:expression} one finally finds that for the potential and transport terms:
\begin{align}  \label{interior:bd:potentialtransport}
        & - 2\int_0^{z^*} ( \frac{\lambda_s}{\lambda} + \frac{\nu_s}{\nu}  + \phi +\varepsilon)   w \varepsilon_z^2 dz +2\int_0^{z^*} (\partial_z^{-1} \phi - \frac{\nu_s}{\nu} z + \partial_z^{-1} \varepsilon) \varepsilon_{zz}  w \varepsilon_z dz\\
 \nonumber       \geq &\left( \delta + \frac\eta2+O(e^{-\frac{\delta}{2}s})\right)\mathcal E_1(s).
\end{align}

\noindent \underline{Nonlocal terms}. Notice that since $\varepsilon(z=0)=0$, by the 
Cauchy-Schwarz inequality, one obtains that for all $z\in [0,z^*]$:
\begin{align}\label{bound:epsilon}
  &  \varepsilon(z) \leq \mathcal E_1  ( \int_0^{z} \frac{1}{w}(\widetilde{z}) d\widetilde{z})^{\frac{1}{2}} , \\
\label{bound:primitive-of-epsilon}
&   \partial_z^{-1}\varepsilon(z) =\int_0^z \varepsilon(\widetilde{z}) d\widetilde{z} \leq \mathcal E_1 \int_0^z (\int_0^{\widetilde{z}} \frac{1}{w}(\xi) d\xi)^{\frac{1}{2}} d\widetilde{z}.
\end{align}
For $K>1$ large enough so that 
\begin{equation}\label{condition:K}
    \frac{\ln K}{K} \leq K^{-\frac23},
\end{equation}
we introduce the parameter $\rho=\min(\frac13,\frac1{\beta+1}),$ and, thanks to \eqref{bound:epsilon}, \eqref{bound:primitive-of-epsilon} and Lemma \ref{lemma:weight-estimate}, by the 
Cauchy-Schwarz inequality, one obtains 
\begin{equation}\label{estimate:linear-1}
   2\int_0^{z^*} \Big| \phi' \varepsilon  w \varepsilon_z \Big| dz \leq 2\Big(\int_0^{z^*} |\phi'|^2 w ( \int_0^{z} \frac{1}{w}(\widetilde{z}) d\widetilde{z}) dz\Big)^{\frac{1}{2}} \mathcal E_1^2 \leq \frac{C(\beta,\eta)}{K^\rho}\mathcal E_1^2,
\end{equation}
where the constant $C(\beta,\eta)$ depends on $\eta$ through $\alpha$, and 
\begin{equation}\label{estimate:linear-2}
    2\int_0^{z^*} \Big| \partial_z^{-1} \varepsilon   w \varepsilon_z \phi'' \Big| dz \leq 2\Big(\int_0^{z^*} |\phi''|^2 w(\int_0^z (\int_0^{\widetilde{z}} \frac{1}{w}(\xi)d\xi )^{\frac{1}{2}} d\widetilde{z} )^2 dz \Big)^{\frac{1}{2}} \mathcal E_1^2 \leq \frac{C(\beta,\eta)}{K^\rho} \mathcal E_1^2.
\end{equation}

\noindent \underline{Source terms}.  Finally, we consider 
\begin{equation}\label{estimate:nonlocal-1}
    4\nu \Big(\int_0^{\frac{1}{\nu}} (\phi + \varepsilon)^2 (z) dz\Big) \int_0^{z^*}  \Big(\beta \phi' + (\beta+1) z\phi''\Big) w \varepsilon_z dz.
\end{equation}
Using \eqref{bd:source} one obtains, for $C$ independent of the bootstrap constants:
\begin{equation}\label{condition:beta-1}
\begin{split}
     4\nu \Big(\int_0^{\frac{1}{\nu}} (\phi + \varepsilon)^2 (z) dz\Big) \leq &  C e^{-(\frac{\delta}{2}+\frac \eta 4) s }. 
\end{split}
\end{equation}
Thanks to Proposition \ref{proposition:profile}, one has the crucial cancellation at the origin
\begin{equation}\label{condition:beta-2}
 |   \beta \phi' + (\beta+1) z\phi'' |\lesssim  z^{\frac2{\beta+1}-1} \quad \text{as } z\rightarrow 0.
\end{equation}
Since $w=z^\alpha e^{-Kz}$ with $\alpha = \frac{|1-\beta|-2+\frac\eta2}{\beta+1}$, one then has using Cauchy-Schwarz
\begin{equation}\label{condition:beta-3}
   \int_0^{z^*}  \Big(\beta \phi' + (\beta+1) z\phi''\Big)  w \varepsilon_z dz \leq \Big(\int_0^{z^*} \Big(\beta \phi' + (\beta+1) z\phi''\Big)^2 w(z) dz\Big)^{\frac12} \mathcal E_1 \leq C({\beta,\eta,z^*,K}) \mathcal E_1.
\end{equation}
Therefore, using \eqref{condition:beta-1} and \eqref{condition:beta-3}, by Young's inequality $xy\leq K^\rho x^2/2+y^2/2K^\rho$, one has:
\begin{equation}\label{estimate:nonlocal-2}
\begin{split}
      &4\nu \Big(\int_0^{\frac{1}{\nu}} (\phi + \varepsilon)^2 (z) dz\Big) \int_0^{z^*}  \Big(\beta \phi' + (\beta+1) z\phi''\Big)  w \varepsilon_z dz \leq C({\beta,\eta,z^*,K}) e^{-(\delta+\frac \eta 2) s} + \frac{\mathcal E_1^2}{{K^\rho}}.
\end{split}
\end{equation}

\noindent \underline{Conclusion}. Injecting \eqref{interior:bd:potentialtransport}, \eqref{estimate:linear-1}, \eqref{estimate:linear-2} and \eqref{estimate:nonlocal-2} in \eqref{interior:id:expression2} shows:
\begin{equation}\label{estimate:together-1}
   \frac{d}{ds}\mathcal E_1^2 + \Big(\delta + \frac\eta2 - \frac{C(\beta,\eta)}{K^\rho} - C({\beta,\eta,K,z^*,\tilde K})e^{-\frac\delta2 s} \Big)\mathcal E_1^2  \leq C({\beta,\eta,z^*,K})e^{-(\delta+\frac \eta 2) s} .
\end{equation}
Choosing $K$ large enough so that 
\begin{equation}\label{condition:K-2}
    \frac{C(\beta,\eta)}{K^\rho}\leq\frac\eta8,
\end{equation}
and $s_0$ large enough so that 
\begin{equation}\label{condition:s0-3}
    C({\beta,\eta,K,z^*,\tilde K})e^{-\frac\delta2 s_0}\leq\frac\eta8,
\end{equation}
we have that \eqref{estimate:together-1} becomes
\begin{equation}\label{differential-inequality:E1}
    \frac{d}{ds}(e^{(\delta+\frac{\eta}{4}) s}\mathcal E_1^2)\leq C({\beta,\eta,z^*,K}) e^{-\frac{\eta}{4} s} .
\end{equation}
Integrating between $s_0$ and $s$ \eqref{differential-inequality:E1} using \eqref{bd:eini} gives 
\begin{equation*}
\begin{split}
    \mathcal E_1^2&\leq e^{(\delta+\frac{\eta}{4}) (s_0-s)}\mathcal E_1^2(s_0)+C({\beta,\eta,z^*,K})  e^{-(\delta+\frac{\eta}{4}) s} <\gamma e^{-\delta s}+C({\beta,\eta,z^*,K})  e^{-(\delta+\frac{\eta}{4}) s} <2 \gamma e^{-\delta s} ,
\end{split}
\end{equation*}
where we have chosen $s_0$ large enough so that 
\begin{equation}\label{condition:s0-4}
    C({\beta,\eta,z^*,K})  e^{-\frac\eta4 s_0} < \gamma.
\end{equation}
This is \eqref{bd:int e1}.

\end{proof}

\begin{lemma}[Exterior Estimate]\label{lemma:main}
For any $0<\eta<\min (\beta,1)$ there exists $\bar z^*\geq 1$, such that the following holds true for all $z^*\geq \bar z^*$. For any $K\geq 1$ there exists $\gamma^*>0$ such that for all $0<\gamma \leq \gamma^*$, for all $\tilde K\geq 1$ and $\tilde\nu_0^*>0$, for $s_0^*$ large enough, for any $s_0\geq s_0^*$, for any solution which is trapped on $[s_0,s_1]$ we have for $s\in [s_0,s_1]$:
\begin{eqnarray}\label{bd:int e2}
\mathcal E_2^2(s)\leq  e^{-\delta s}.
\end{eqnarray}

\end{lemma}

\begin{proof}

We use a standard comparison principle for transport operators, together with a bootstrap argument to control nonlocal effects. To implement this bootstrap argument, we assume in addition that on $[s_0,s_1]$ there holds:
\be \label{exteriormathcalEbootstrap}
\mathcal E_2(s)\leq \tilde K' e^{-\frac \delta 2 s},
\ee
for some constant $0<\tilde K'\leq \tilde K$.
\begin{remark}
 The reason we apply this bootstrap argument here is to make $z^*$ independent of $\tilde K$ (see \eqref{condition:delta*}). This is natural since $z^*$, the boundary between the interior region and the exterior region, should be independent of the size of the remainder which is $\tilde K$.
\end{remark}
Clearly \eqref{exteriormathcalEbootstrap} is satisfied if one chooses $\tilde K'=\tilde K$ from \eqref{bd:etrap}. By applying the modulations \eqref{equation:modulation-1} and \eqref{equation:modulation-2}, we can rewrite \eqref{equation:epsilon-1} as
\begin{equation}\label{equation:epsilon-2}
    \begin{split}
        &\varepsilon_s +\mathcal L \varepsilon    = F ,
    \end{split}
    \end{equation}
where $\mathcal L$ is the transport operator (note that it has a nonlinear part):
$$
\mathcal L v=  - \frac{\lambda_s}{\lambda}v - \frac{\nu_s}{\nu} z v_z - 2\phi v + \partial_z^{-1} \phi v_z    - \varepsilon v + \partial_z^{-1} \varepsilon v_z 
$$
and the source term is
$$
F = -\partial_z^{-1} \varepsilon \phi' +2\nu\Big(\int_0^{\frac{1}{\nu}} (\phi+\varepsilon)^2(z)dz\Big)\Big(-1+\phi -(\beta+1)z\phi'\Big).
$$

\noindent \textbf{Step 1}. \emph{A supersolution for $\pa_s+\mathcal L$ on $[z^*,\nu^{-1}]$}. We introduce
$$
f(s,z)= e^{-\frac{\delta}{2}s}
$$
and claim that there exists $z^*$ large enough such that for $s_0$ large enough, for all $ s_0\leq s \leq s_1$ and $z\geq z^*$:
\be \label{exterior:id:supersolution}
(\pa_s +\mathcal L) f\geq \frac{\eta}{4} e^{-\frac{\delta}{2}s}.
\ee
We now prove \eqref{exterior:id:supersolution}. We compute using $\pa_z f=0$, \eqref{modulationequations}, $\phi\lesssim z^{-1/\beta}$ as $z\rightarrow \infty$ and \eqref{bd:epsilon}:
\begin{align*}
(\pa_s +\mathcal L) f & = \left(-\frac{\delta}{2} - \frac{\lambda_s}{\lambda} - 2\phi  - \varepsilon \right)f\\
&=\left(-\frac{\delta}{2} +1+O(e^{-\frac \delta 2 s})+O(z^{-\frac 1 \beta})+O(\sqrt{C^* } \tilde K e^{-\frac \delta 2}s)\right)e^{-\frac \delta 2 s}\\
&\geq \left(\frac{\eta}{2}+O(e^{-\frac \delta 2 s})+O(z^{-\frac 1 \beta})+O(C(\tilde K,z^*,\beta,\eta)e^{-\frac \delta 2}s)\right) e^{-\frac \delta 2 s} 
\end{align*}
where we used $\delta=2\min(\beta,1)-\eta$. This implies \eqref{exterior:id:supersolution} upon choosing $z^*$ large enough and then $s_0$ large enough.\\

\noindent \textbf{Step 2}. \emph{Estimate for the source term}. We claim that for all $z\geq z^*$ and $s_0\leq s \leq s_1$, we have 
\be \label{exterior:bd:source}
|F(s,z)|\leq \frac{\eta}{8}\left(\frac 14+\frac{\tilde K'}{4}\right)e^{-\frac{\delta }{2}s} .
\ee
To prove this inequality, observe that for $0\leq z\leq z^*$ we have using \eqref{bd:epsilon:interiorzone} and \eqref{bd:int e1} that
\be \label{exterior:bd:interiorvarepsilon}
|\varepsilon (z)|\lesssim \sqrt{2C^* \gamma} e^{-\frac{\delta}{2}s}
\ee
while for $z\geq z^*$ one has $|\varepsilon(z)|\leq \tilde K' e^{-\frac \delta 2 s}$ thanks to \eqref{exteriormathcalEbootstrap}. Therefore, thanks to the behavior of $\phi'$ near $\infty$, one has
\begin{equation} \label{exterior:bd:inter1}
\begin{split}
    |\partial_z^{-1} \varepsilon \phi'| &\leq \|\varepsilon\|_{L^\infty} z |\phi'(z)| \leq C\Big(\sqrt{2C^* \gamma}e^{-\frac \delta 2 s}+\tilde K' e^{-\frac{\delta }{2}s} \Big)z^{-\frac 1 \beta} \leq \frac{\eta}{16}\left(\frac 14+\frac{\tilde K'}{4}\right)e^{-\frac{\delta }{2}s} 
\end{split}
\end{equation}
where we chose $z^*$ large enough and then $\gamma$ small enough such that:
\begin{equation}\label{condition:delta*}
C\sqrt{2C^* \gamma}z^{*-\frac 1 \beta} \leq  \frac{\eta}{64 } \qquad \mbox{and} \qquad C  z^{*-\frac 1 \beta} \leq \frac{\eta}{64}.
\end{equation}
From Proposition \ref{proposition:profile}, we know that $\Big|-1+\phi -(\beta+1)z\phi'\Big|\leq C$ is uniformly bounded. Therefore using \eqref{bd:source} we get:
\be \label{exterior:bd:inter2}
\left| 2\nu\Big(\int_0^{\frac{1}{\nu}} (\phi+\varepsilon)^2(z)dz\Big)\Big(-1+\phi -(\beta+1)z\phi'\Big)\right| \lesssim e^{-(\frac \delta 2+\frac \eta 4)s}\leq \frac{\eta}{64}e^{-\frac \delta 2 s}
\ee
for $s_0$ large enough. Summing \eqref{exterior:bd:inter1} and \eqref{exterior:bd:inter2} implies \eqref{exterior:bd:source}.\\

\noindent \textbf{Step 3}. \emph{Applying the comparison principle}. Let
\be \label{exterior:def:fpm}
f^\pm= \pm \left( (\frac 14+\frac{\tilde K'}{4})f-\varepsilon \right).
\ee
Then using \eqref{equation:epsilon-2}, \eqref{exterior:id:supersolution} and \eqref{exterior:bd:source} one computes that for $s\leq s\leq s_1$ and $z\in [z^*,\nu^{-1}]$:
$$
(\pa_s+\mathcal L)f^+=(\frac 14+\frac{\tilde K'}{4}) (\pa_s+\mathcal L)f+F\geq (\frac 14+\frac{\tilde K'}{4})\frac{\eta}{4} e^{-\frac{\delta}{2}s}-\frac{\eta}{8}\left(\frac 14+\frac{\tilde K'}{4}\right)e^{-\frac{\delta }{2}s}\geq 0
$$
Similarly, $(\pa_s+\mathcal L)f^-\leq 0$. Recall that from \eqref{particle-moving} we know that the particles are always moving from region $0\leq z\leq z^*$ to region $z^*\leq z \leq \frac1\nu$. At the boundary $z=z^*$ one has using \eqref{exterior:bd:interiorvarepsilon} that:
$$
f^+(s,z^*)= (\frac 14+\frac{\tilde K'}{4})e^{-\frac{\delta}{2}s}-\varepsilon (s,z^*)\geq (\frac{1}{4}-\sqrt{2C^* \gamma})e^{-\frac{\delta}{2}s} \geq 0
$$
provided that
\begin{equation}\label{condition:delta*onemore}
C\sqrt{2C^* \gamma} \leq \frac 14.
\end{equation}
Similarly, $f^-(s,z^*)\leq 0$. At the point $z(s)=\frac1{\nu(s)}$, the characteristics of the full transport field stays on the boundary since
\begin{equation}\label{partical-stay}
    \frac{d}{ds}\frac1{\nu(s)}= -\frac{\nu_s}{\nu^2} = -\frac{\nu_s}{\nu^2} + \int_0^{\frac1\nu} (\phi+\varepsilon) dz,
\end{equation}
where we have used \eqref{equation:epsilon-bc}. Thanks to \eqref{bd:eini}, we know initially
\begin{equation*}
    f^+(s_0,\frac1{\nu_0}) =  (\frac 14+\frac{\tilde K'}{4})e^{-\frac{\delta}{2}s_0}-\varepsilon (s_0,\frac1{\nu_0})\geq (\frac 14+\frac{\tilde K'}{4})e^{-\frac{\delta}{2}s_0} - \frac14 e^{-\frac{\delta}{2}s_0} \geq 0,
\end{equation*}
thus $f^+(s,\frac1{\nu(s)}) \geq 0$ for all $s\geq s_0.$ Similarly $f^-(s,\frac1{\nu(s)}) \leq 0$. At initial time $s=s_0$, we have $f^+(s_0,z)\geq 0 $ because of \eqref{bd:eini}, and similarly $f^-(s_0,z)\leq 0 $. Therefore, one can apply the maximum principle and obtain that $f^+(s,z)\geq 0$ and $f^-(s,z)\leq 0$ for all $s_0\leq s \leq s_1$ and $z^*\leq z \leq \nu^{-1}$. By their definition \eqref{exterior:def:fpm} this implies:
\be \label{exterior:bd:inter5}
|\varepsilon (s,z)|\leq (\frac{1}{4}+\frac{\tilde K'}{4})e^{-\frac{\delta}{2}s} \qquad \mbox{for all }s_0\leq s \leq s_1 \mbox{ and }z^*\leq z \leq \nu^{-1}.
\ee

\noindent \textbf{Step 4}. \emph{End of the proof}. We first set $\tilde K_0'=\tilde K$ so that \eqref{exteriormathcalEbootstrap} is satisfied with constant $\tilde K'=\tilde K'_0$ because of \eqref{exteriormathcalEbootstrap}. Then, we obtain \eqref{exterior:bd:inter5} with constant $\tilde K'=\tilde K'_0$. This implies that \eqref{exteriormathcalEbootstrap} is satisfied with constant $\tilde K'=\tilde K'_1$ given by $\tilde K'_1=\varphi (\tilde K_0')$ with $\varphi (\tilde K')=(1+\tilde K')/4$. We iterate this procedure, and obtain constants $\tilde K_2$, then $\tilde K_3$, ... , $\tilde K_n$ such that  \eqref{exteriormathcalEbootstrap} is satisfied with constants $\tilde K_2$, then $\tilde K_3$, ... , $\tilde K_n$ with $\tilde K_{k+1}=\varphi(\tilde K_{k})$. By iterating a finite (depending on $\tilde K$) number of times $k$ we obtain $\tilde K'_k\leq 1$ and \eqref{exteriormathcalEbootstrap} then implies \eqref{bd:int e2}.

\end{proof}

We can now end the proof of Proposition \ref{pr:bootstrap}.

\begin{proof}[Proof of Proposition \ref{pr:bootstrap}]

We set $\tilde K=3$ (or any $\tilde K >1$). For any $\beta>0$ and $0<\eta<\min (1,\beta)$, we pick $z^*\geq \bar z^*$ where $\bar z^*(\beta,\eta)$ is given by Lemma \ref{lemma:main}, then we pick $K\geq K^*$ where $K^*(\beta,\eta,z^*)$ is given by Lemma \ref{lemma:interior}, and then we pick $0<\gamma\leq \min (\gamma^*,1)$ where $\gamma^*(\beta,\eta,z^*,K^*)$ is given by Lemma \ref{lemma:main}.

Then for any $\tilde \nu_0^*>0$, there exists $s_0^*$ such that the conclusions of Lemmas \ref{lemma:modulation}, \ref{lemma:interior} and \ref{lemma:main} are simultaneously valid for all $s_0\geq s_0^*$.

For such choices of constants, consider an initial data $a_0$ trapped in the sense of Definition \ref{definition:initial-subcritical}. We define:
$$
s^*=\sup \ \{s_1\geq s_0, \ a \mbox{ is trapped in the sense of Definition \ref{definition:trap-subcritical} on } [s_0,s_1]\}.
$$
If $s^*=\infty$, then Proposition \ref{pr:bootstrap} is proved. We assume by contradiction $s^*<\infty$. Then, applying Lemmas \ref{lemma:modulation}, \ref{lemma:interior} and \ref{lemma:main} we obtain at time $s^*$:
$$
 \frac 12 e^{-s^*}\leq \lambda(s^*) \leq \frac 32 e^{-s^*}, \quad \frac{\tilde \nu_0}2 e^{-\beta s^*}\leq \nu(s^*) \leq \frac 32 \tilde \nu_0 e^{-\beta s^*}, \quad \mathcal E_1^2(s^*)  \leq 2\gamma e^{-\delta s^*}, \quad \mathcal E_2^2(s^*)\leq  e^{-\delta s^*}.
$$
Since $\tilde K=3$, the bounds of Definitions \ref{definition:trap-subcritical} are thus strictly satisfied at time $s^*$, and by a continuity argument, there exists $s'>s^*$ such that these bounds are satisfied on $[s^*,s']$. But this contradicts the definition of $s^*$. Hence $s^*=\infty$ and Proposition  \ref{pr:bootstrap} is proved.

\end{proof}

\begin{remark}
 Notice that when $\eta$ getting closer to $0$, one can get better decay. However, the decay rate of $\mathcal E_1$ and $\mathcal E_2$ cannot reach or be faster than $e^{-\beta s}$ when $0<\beta\leq 1$ and $e^{-s}$ when $\beta> 1$. The reason is:
 \begin{itemize}
     \item For $0< \beta \leq 1$, the decay of $\nu$ is only $e^{-\beta s}$.
     \item For $\beta>1$, in order to make \eqref{condition:beta-3} integrable near $0$, one needs $\frac4{\beta+1}-2+\alpha > -1 \Leftrightarrow \alpha > \frac{\beta-3}{\beta+1}$. Such restriction on $\alpha$ makes the spectral gap in \eqref{spectral-gap} strictly less than $2$.
 \end{itemize}
  
\end{remark}

We can now end the proof of Theorem \ref{nonsmooth}.

\begin{proof}[Proof of Theorem \ref{nonsmooth}.]

Pick $\beta>0$. We write $\phi=\phi_\beta$ in the proof for simplicity. Choose then any $0<\eta<\min (\beta,1)$ and let the constants $\tilde K>1$, $K,z^*\geq 1$ and $\gamma>0$ be given by Proposition \ref{pr:bootstrap}. For a fixed $\tilde \nu_0^*>0$, and let $s_0^*$ be given by Proposition \ref{pr:bootstrap} and define $\lambda_0^*=e^{-2s_0^*}$.

Let then, with a parameter $\kappa$ to be fixed later on in the proof:
\be \label{rough:id:defparameters}
\lambda_0\leq \lambda_0^*/2 \quad  \mbox{and} \quad \tilde \nu_0\leq \tilde\nu_0^*/2
\ee
and an initial datum $a_0$ of the form \eqref{nonsmooth:id:initial} satisfying \eqref{nonsmooth:id:initialcond}.\\

\noindent \textbf{Step 1}. \emph{Proof assuming a claim}. We claim that for $\kappa$ small enough, $a_0$ is initially trapped in the sense of Definition \ref{definition:initial-subcritical}, with framework parameters $\eta,\tilde K,K,z^*,\gamma,\lambda_0^*,\tilde \nu_0^*$ defined right above, and decomposition parameters $\bar \lambda_0,\bar \nu_0$ to be determined in Step 2. Assuming this claim, we have using Proposition \ref{pr:bootstrap} that the solution is trapped in the sense of Definition \ref{definition:trap-subcritical} for all $s\in [\bar s_0,\infty)$ where $\bar s_0=\log \bar \lambda_0^{-1}$.

We unwind the self-similar transformation \eqref{id:self-similarvariables} using \eqref{bd:boostrap improved parameters2} and define $T=\int_{\bar s_0}^\infty \lambda(s)ds<\infty$ so that:
$$
t(s)=\int_{\bar s_0}^s \lambda(\tilde s)d\tilde s=T-\int_{s}^\infty \tilde\lambda_\infty e^{-\tilde s}(1+O(e^{-\frac{\delta}{2}\tilde s}))d\tilde s=T- \tilde\lambda_\infty e^{-s}+O(e^{-(1+\frac{\delta}{2})s})
$$
and hence $\tilde\lambda_\infty e^{-s}=(T-t)+O((T-t)^{1+\delta/2})$. We then get using \eqref{bd:epsilon} and \eqref{bd:boostrap improved parameters2} that 
\begin{equation*}
    \| \varepsilon (s)\|_{L^\infty(0,\nu^{-1})}\leq C (T-t)^{\delta/2}, \ \ \lambda(s)=T-t+O((T-t)^{1+\delta/2}),\ \ \nu=\frac{\tilde \nu_\infty}{\tilde \lambda_\infty^\beta} (T-t)^{\beta}+O((T-t)^{\beta+\delta/2}).
\end{equation*}
Back in original variables \eqref{decomp}, this implies the desired results \eqref{nonsmooth:1} and \eqref{nonsmooth:2} in the Theorem by renaming $\delta$ as $2\delta$ and $\tilde \nu_\infty$ as $\tilde \nu_\infty \tilde \lambda_\infty^\beta$.\\

\noindent \textbf{Step 2}. \emph{Proof of the claim}. For $\bar \lambda_0,\bar \nu_0>0$ we define $\bar a_0$ and $\bar \varepsilon_0$ as:
$$
a_0(Z)=\frac{1}{\lambda_0}\phi \left(\frac{Z}{\lambda_0^\beta \tilde \nu_0}\right)+\tilde a_0(Z)=\frac{1}{\bar \lambda_0}\phi \left(\frac{Z}{\bar \lambda_0^\beta \bar \nu_0}\right)+\bar a_0(Z), \  \ \bar \varepsilon_0 (z)=\bar \lambda_0 \bar a_0(\bar \lambda_0^\beta \bar \nu_0 z), \  \ Z = z\bar \lambda_0^\beta \bar \nu_0.
$$
Then, introducing $\mu=\bar \lambda_0^\beta \bar \nu_0 \lambda_0^{-\beta}\tilde \nu_0^{-1} $ we have the two decompositions for $0<z\leq 1$ ($\varepsilon_1$ to $\varepsilon_4$) and $1\leq z \leq \bar \lambda_0^{-\beta}\bar \nu_0^{-1}$ ($\bar \varepsilon_1$ and $\bar \varepsilon_2$):
\begin{align}
\label{rough:defdecomposition} &\bar \varepsilon_0 (z) \ = \ \bar \varepsilon_1 (z)+\bar \varepsilon_2(z) \ = \ \varepsilon_1(z)+ \varepsilon_2(z)+\varepsilon_3(z)+\varepsilon_4(z),\\
\label{rough:defbar} & \bar \varepsilon_1 (z)=\frac{\bar \lambda_0}{\lambda_0} \phi (\mu z)-\phi(z), \quad \bar \varepsilon_2(z)=\bar \lambda_0 \tilde a_0(\bar\lambda_0^\beta \bar \nu _0 z),\\
\label{rough:def0} & \varepsilon_1(z)= \frac{\bar \lambda_0}{\lambda_0}-1+\bar \lambda_0 \tilde a_0(0)+z^{\frac{1}{\beta+1}}(1-\frac{\bar \lambda_0}{\lambda_0}\mu^{\frac{1}{\beta+1}}),  \quad \varepsilon_2(z)=\left(\frac{\bar \lambda_0}{\lambda_0}-1\right)(\phi (\mu z)-1+(\mu z)^{\frac{1}{\beta+1}}),\\
\nonumber &\varepsilon_3(z)=\phi (\mu z)-\phi(z)+(\mu^{\frac{1}{\beta+1}}-1)z^{\frac{1}{\beta+1}}, \quad \varepsilon_4=\bar \lambda_0 (\tilde a_0(z\bar \lambda_0^\beta \bar \nu_0)-\tilde a_0(0)).
\end{align}
In order for the boundary condition \eqref{orthogonality} to be satisfied, using the behavior \eqref{asympt z to 0} of $\phi(z)$ as $z\to 0$, we require that $\varepsilon_1=0$. Using \eqref{rough:def0} this fixes $\bar \lambda_0,\bar \nu_0$ in an unique manner via the identities:
\be \label{rough:bd:parameter}
\bar \lambda_0=(\lambda_0^{-1}+\tilde a_0(0))^{-1}=\lambda_0 (1+O(\kappa \bar \lambda_0)), \qquad \bar \nu_0=\lambda_0^{1+2\beta}\bar \lambda_0^{-1-2\beta}\tilde \nu_0=\tilde \nu_0 (1+O(\kappa \bar \lambda_0))
\ee
where we used \eqref{nonsmooth:id:initialcond}. Injecting \eqref{rough:bd:parameter} in \eqref{rough:defbar}, using \eqref{asympt z to infty} one then obtains that for all $z\geq 1$:
$$
|\bar \varepsilon_1(z)|\lesssim |\frac{\bar \lambda_0}{\lambda_0}-1|\phi(\mu z)+ |\mu z-z|\sup_{\tilde z \in [z,\mu z]}|\phi'(\tilde z)|\lesssim \kappa \bar \lambda_0 z^{-\frac{1}{\beta}}, \quad \mbox{and similarly}\quad |\pa_z \bar\varepsilon_1 (z)|\lesssim \kappa \bar \lambda_0 z^{-\frac{1}{\beta}-1}.
$$
Using \eqref{nonsmooth:id:initialcond} we have for $z\geq 1$ that $|\bar \varepsilon_2(z)|\leq \bar \lambda_0\kappa$ and $|\pa_z \varepsilon_2(z)|\leq \kappa \bar \lambda_0^{1+\beta}\bar \nu_0$. Injecting these two inequalities and the two above in the first decomposition in \eqref{rough:defdecomposition} shows:
\be \label{rough:bd:inter1}
|\bar \varepsilon_0(z)|\lesssim \kappa \bar \lambda_0 \quad \mbox{and} \quad |\pa_z  \bar\varepsilon_0 (z)|\lesssim \kappa \bar \lambda_0 (z^{-1-\frac{1}{\beta}}+ \bar \lambda_0^{\beta}\bar \nu_0) \qquad \mbox{for all }z\in [1,\bar \lambda_0^{-\beta}\bar \nu_0^{-1}]
\ee
Next, using the behavior \eqref{asympt z to 0} of $\phi(z)$ as $z\rightarrow 0$ and then \eqref{rough:bd:parameter} we deduce that for $0<z\leq 1$:
$$
|\pa_z \varepsilon_2 (z)|\lesssim |\frac{\bar \lambda_0}{\lambda_0}-1|z^{\frac{2}{\beta+1}-1} \lesssim \kappa \bar \lambda_0z^{\frac{2}{\beta+1}-1} .
$$
By a similar estimate, $|\pa_z \varepsilon_3 (z)| \lesssim \kappa \bar \lambda_0z^{\frac{2}{\beta+1}-1}$ for $0<z\leq 1$. Using \eqref{nonsmooth:id:initialcond} we obtain $|\pa_z \varepsilon_4|\lesssim \kappa \bar \lambda_0^{1+\beta}\bar \nu_0$. Injecting these inequalities and the above one in the second decomposition and $\varepsilon_1=0$ in \eqref{rough:defdecomposition} shows:
\be \label{rough:bd:inter2}
|\pa_z \bar \varepsilon_0(z)|\lesssim \kappa \bar \lambda_0 z^{\frac{2}{\beta+1}-1}+\kappa \bar \lambda_0^{1+\beta}\bar \nu_0  \qquad \mbox{for all }z\in (0,1] .
\ee
Combining \eqref{rough:bd:inter1} and \eqref{rough:bd:inter2}, using $w=z^\alpha e^{-Kz}$ with $\alpha = \frac{|1-\beta|-2+\frac\eta2}{\beta+1}>-1$ we obtain:
\be \label{rough:bd:inter3}
\int_0^{z^*} w \bar \varepsilon_0^2 dz + \sup\limits_{z^*\leq z\leq \bar\lambda_0^{-\beta} \bar\nu_0^{-1}} \bar \varepsilon_0^2(z) \lesssim \kappa^2 \bar \lambda_0^2 +\kappa^2 \bar \lambda_0^{2+2\beta}\bar \nu_0^2.
\ee
We now check that $a_0$ is initially trapped in the sense of Definition \ref{definition:initial-subcritical} with decomposition parameters $\bar \lambda_0$ and $\bar \nu_0$, and framework parameters $\eta,\tilde K,K,z^*,\gamma,\lambda_0^*,\tilde \nu_0^*$ defined right before Step 1. We set $\bar s_0=\log \bar \lambda_0^{-1}$. The estimates \eqref{rough:id:defparameters} and \eqref{rough:bd:parameter} imply \eqref{bd:parametersini} for $\bar \lambda_0,\bar s_0,\bar \nu_0,\tilde\nu_0^*$, so that item (i) of Definition \ref{definition:initial-subcritical} is indeed satisfied. The fact that $\varepsilon_1=0$ in the second decomposition in \eqref{rough:defdecomposition} and the inequality \eqref{rough:bd:inter2} show that item (ii) of Definition \ref{definition:initial-subcritical} is satisfied. Finally, \eqref{rough:bd:inter3} shows that item (iii) of Definition \ref{definition:initial-subcritical} is also satisfied provided $\kappa$ has been chosen small enough depending only on $\lambda_0^*$ and $\tilde \nu_0^*$. Hence $a_0$ is initially trapped, finishing the proof of the claim.

\end{proof}

\section{The smooth blowup case}\label{sec:smooth}
In this section, we prove Theorem \ref{theorem:critical}. We study the limiting critical case when $\beta=0$, for which
$$
\phi_{\beta=0}(z)=\phi(z)=e^{-z}
$$
(we drop the $\beta$ subscript in this section to ease notation). When $\beta=0$, the vanishing condition \eqref{orthogonality} becomes:
\be \label{smooth:orthogonality}
\left\{ \begin{array}{l l}  \varepsilon(s,z=0)=0 ,\\
 \pa_z \varepsilon(s,z=0)= 0, \end{array} \right.
\ee
and the modulation equations \eqref{equation:modulation-1} and \eqref{equation:modulation-2} become:
\begin{equation} \label{smooth:id:modulation}
    \frac{\lambda_s}{\lambda} +1 = -\frac{\nu_s}{\nu} = 2\nu \int_0^{\frac{1}{\nu}} (\phi + \varepsilon)^2 (z) dz.
\end{equation}
Therefore, one can rewrite \eqref{equation:epsilon-1} and \eqref{equation:epsilon_z} as
\begin{equation}  \label{smooth:id:equationep}
    \begin{split}
        &\varepsilon_s - \frac{\lambda_s}{\lambda}\varepsilon - \frac{\nu_s}{\nu} z \varepsilon_z - 2\phi \varepsilon + \partial_z^{-1} \phi \varepsilon_z + \partial_z^{-1} \varepsilon \phi' 
         - \varepsilon^2 + \partial_z^{-1} \varepsilon \varepsilon_z 
        \\
        =& 2\nu\Big(\int_0^{\frac{1}{\nu}} (\phi+\varepsilon)^2(z)dz\Big)\Big(-1+(z+1)\phi\Big),
    \end{split}
\end{equation}
and
\begin{equation} \label{smooth:id:equationpazep}
\begin{split}
     &\varepsilon_{zs} - ( \phi -1  )  \varepsilon_{z}      + (\partial_z^{-1} \phi - \frac{\nu_s}{\nu} z ) \varepsilon_{zz} - \phi' \varepsilon + \partial_z^{-1} \varepsilon \phi''
     - \varepsilon \varepsilon_z + \partial_z^{-1} \varepsilon \varepsilon_{zz}
    \\
    = & -  z \phi 2\nu \int_0^{\frac{1}{\nu}} (\phi + \varepsilon)^2 (z) dz,
\end{split}
\end{equation}

The proof of the theorem follows the same strategy as that of Theorem \ref{nonsmooth}. It also relies on a bootstrap argument. However, \eqref{smooth:id:modulation} gives $\nu_s =-\nu^2$ to leading order, hence we will have to deal with the slower algebraic decay $\nu\approx s^{-1}$ in comparison with the exponential decays involved in the proof of Theorem \ref{nonsmooth}. We first need to adjust Definition \ref{definition:initial-subcritical} and Definition \ref{definition:trap-subcritical}. 

We consider the weight:
\be \label{smooth:def:w}
w(z) = z^{-2}.
\ee
Here, since $\phi(z)=e^{-z}$, explicit computations to control nonlocal terms will avoid the use of a $e^{-Kz}$ factor in the weight. 

\begin{definition}[Initial closeness]\label{def:ini crit-beta=0}
Let $\lambda_0^*>0$ and $\gamma>0$. We say that $a_0$ is initially close to the blow-up profile if there exists $\lambda_0>0$ and $\nu_0>0$ such that the decomposition \eqref{decomp} satisfies:
\begin{itemize}
\item[(i)] \emph{Initial values of the modulation parameters} (note that this fixes the value of $s_0$):
\begin{eqnarray}
\label{bd:parametersini-beta=0}
 \lambda_0 = s_0e^{-s_0}, \ \  \frac{1}{2s_0}\leq \nu_0 \leq  \frac2{s_0}.
\end{eqnarray} 
\item[(ii)] \emph{Compatibility condition for the initial perturbation}. $\varepsilon_0\in C^2 ([0,\frac1{\nu_0}))$ satisfies the boundary conditions \eqref{smooth:orthogonality} and the integral condition \eqref{equation:epsilon-bc}.
\item[(iii)] \emph{Initial smallness of the remainder in the self-similar variables}. For some small number $\gamma > 0$, with $w$ given by \eqref{smooth:def:w}:
\begin{eqnarray}\label{bd:eini-beta=0}
\mathcal E_1^2(s_0) = \int_0^{z^*} w\varepsilon_0^2\;dz  < \gamma^2 s_0^{-\frac43},
\ \ \mathcal E_2^2(s_0) =\sup_{z^*\leq z\leq\frac1{\nu_0}}|\varepsilon_0|^2 < \frac14s_0^{-\frac43}.
\end{eqnarray} 
\end{itemize}
\end{definition}

\begin{definition}[Trapped solutions] \label{def:trap crit}

Let $s_0^*\geq 0$, $z^*\geq 1$ and $\gamma>0$. We say that a solution $a(s,z)$ is trapped on $[s_0,s_1]$ with $s_0^*\leq s_0<s_1\leq \infty$, if it satisfies the properties of Definition \ref{def:ini crit-beta=0} at time $s_0$ and if for all $s\in [s_0,s_1]$, $a(s,z)$ can be decomposed as in \eqref{decomp} with:
\begin{itemize}
\item[(i)] \emph{Values of the modulation parameters:}
\begin{eqnarray}\label{bd:parameterstrap-critical}
\frac 1{4} se^{-s}< \lambda < 4 se^{-s}, \ \ \frac 1{4s} 
<\nu < \frac{4}s.
\end{eqnarray} 
\item[(ii)] \emph{Decay in time of the remainder in the self-similar variables:}

\be \label{smooth:bd:etrap}
\mathcal E_1^2(s) = \int_0^{z^*}w\varepsilon_z^2\;dz  < s^{-\frac43}, \qquad \mathcal E_2^2(s) = \sup_{z^*\leq z\leq\frac1{\nu(s)}}|\varepsilon|^2 < s^{-\frac43}.
\ee

\end{itemize}

\end{definition}

\begin{remark}
 One could show that the decay rate for $\mathcal E_1$ and $\mathcal E_2$ is $s^{-1+\eta}$ for any $0<\eta<1$. Here for simplicity we take $s^{-\frac23}$ as an example. 
\end{remark}

The heart of our analysis, as in the case $\beta>0$, is to show that a solution that is initially trapped will remain globally trapped in time self-similar time $s$.

\begin{proposition} \label{smooth:pr:bootstrap}

There exist universal constants $z^*\geq 1$, $\gamma>0$ and $s_0^*\geq 0$ such that the following holds true. For all $s_0\geq s_0^*$, any solution of \eqref{1D-model-intro} which is initially close to the blow-up profile in the sense of Definition \ref{def:ini crit-beta=0} is trapped on $[s_0,+\infty)$ in the sense of Definition \ref{def:trap crit}.

\end{proposition}

The proof of Proposition \ref{smooth:pr:bootstrap} necessitates several lemmas that improve strictly all \textit{a priori} estimates of Definition \ref{def:trap crit}.

\begin{lemma}

For any $z^*\geq 1$, for $s_0^*$ large enough, if $a$ is trapped on $[s_0,s_1]$ then for all $s_0\leq s \leq s_1$:
\be \label{smooth:bd:varepsilonLinfty}
\| \varepsilon (s) \|_{L^\infty([0,\nu^{-1}])} \leq C(z^*) s^{-\frac 23}.
\ee
and
\be \label{smooth:bd:integral}
\nu  \int_0^{\frac{1}{\nu}} (\phi + \varepsilon)^2 (z) dz  \leq 4 s^{-1}.
\ee

\end{lemma}

\begin{proof}

From the vanishing boundary condition, Cauchy-Schwarz, \eqref{smooth:def:w} and \eqref{smooth:bd:etrap}, for $0<z\leq z^*$:
\be \label{smooth:bd:varepsiloninterior}
|\varepsilon(z)|=|\int_0^z \pa_z \varepsilon d\tilde z|\leq \mathcal E_1 \sqrt{\int_0^z z^2 d\tilde z}\lesssim s^{-\frac23} z^{\frac 32}.
\ee
This, combined with the second inequality in \eqref{smooth:bd:etrap}, shows \eqref{smooth:bd:varepsilonLinfty}. Then, since $\phi=e^{-z}$ we estimate:
\begin{align*}
& \int_0^{\frac{1}{\nu}} (\phi + \varepsilon)^2 (z) dz  \leq \int_0^\infty \phi^2 dz +2\int_0^{\nu^{-1}}\phi \varepsilon +\int_0^{\nu^{-1}}\varepsilon^2dz \\
&\qquad \qquad \leq \frac{1}{2}+2\| \varepsilon\|_{L^\infty([0,\nu^{-1})}+\nu^{-1}\| \varepsilon\|_{L^\infty([0,\nu^{-1})}^2 \leq \frac{1}{2}+O(s^{-1/3})\leq 1
\end{align*}
for $s_0^*$ large enough, where we used \eqref{smooth:bd:varepsilonLinfty}. The above inequality and \eqref{bd:parameterstrap-critical} show \eqref{smooth:bd:integral}.

\end{proof}

\begin{lemma}[Modulation Equations] \label{lemma:smoothmodulation}
For any $z^*\geq 1$ and $ \gamma>0$, there exists a large self-similar time $s_0^*$ such that for any $s_0\geq s_0^*$, for any solution which is trapped on $[s_0,s_1]$, we have for $s\in [s_0,s_1]$:
\begin{eqnarray}\label{smoothmodulationequations}
\left| \frac{\lambda_s}{\lambda}+1\right|\leq C s^{-1}, \qquad \left| \frac{\nu_s}{\nu} \right|\leq C s^{-1},
\end{eqnarray}
for $C>0$ independent of the bootstrap constants, and
\begin{eqnarray}\label{smooth:bd:boostrap improved parameters}
 \frac 12 se^{-s}\leq \lambda \leq \frac 32 se^{-s}, \qquad \frac{1}{3s} \leq \nu \leq\frac{3}{s}.
\end{eqnarray}
Moreover, if $s_1=\infty$ then there exists a constant $\tilde \lambda_\infty>0$ such that
\begin{eqnarray}\label{smooth:bd:boostrap improved parameters2}
\lambda=\tilde \lambda_\infty  s e^{-s}\big(1+ O(s^{-\frac 13}) \big) ,\quad
\nu=\frac{1}{s}+O(s^{-\frac 43}).
\end{eqnarray}
\end{lemma}

\begin{proof}

\textbf{Step 1}. \emph{A preliminary estimate}. We claim that:
\be \label{smooth:modulation:inter4}
2\int_0^{\frac{1}{\nu}} (\phi + \varepsilon)^2 (z) dz =1+O(C(z^*)s^{-\frac 13}).
\ee
Indeed, as $\phi(z)=e^{-z}$, we have $2\int_0^\infty \phi^2=1$. Hence, using \eqref{bd:parameterstrap-critical} and \eqref{smooth:bd:varepsilonLinfty}:
\begin{align*}
2\int_0^{\frac{1}{\nu}} (\phi + \varepsilon)^2 (z) dz&= 1+2 \int_{\nu^{-1}}^\infty \phi^2 (z)dz+4 \int_0^{\frac{1}{\nu}} \phi  \varepsilon (z) dz +2\int_0^{\nu^{-1}}  \varepsilon^2 dz \\
&= \ 1+O( e^{-2\nu^{-1}})+O(\| \varepsilon \|_{L^\infty([0,\nu^{-1})]})+O(\nu^{-1}\| \varepsilon \|_{L^\infty([0,\nu^{-1})]}^2) \\
 &= \ 1  +O(C(z^*)s^{-\frac 13}).
\end{align*}

\noindent \textbf{Step 2}. \emph{Equation for $\nu$}. Injecting \eqref{smooth:modulation:inter4} in \eqref{smooth:id:modulation} gives:
\be \label{smooth:modulation:inter2}
-\frac{\nu_s}{\nu^2} = \ 1  +O(C(z^*)s^{-\frac 13}).
\ee
Multiplying \eqref{smooth:modulation:inter2} by $\nu$ and using \eqref{bd:parameterstrap-critical} shows the second inequality in \eqref{smoothmodulationequations}. Integrating \eqref{smooth:modulation:inter2} with time, we find:
\be \label{smooth:modulation:inter3}
\frac{1}{\nu}=\frac{1}{\nu_0}+s-s_0+O(C(z^*)(s^{\frac 23}-s_0^{\frac 23}))=\frac{1}{\nu_0}+s-s_0+O(C(z^*)s^{-\frac 13}(s-s_0)).
\ee
Therefore, since $\nu_0^{-1}\leq 2s_0$ from \eqref{bd:parametersini-beta=0} we infer for $s_0^*$ large enough depending on $z^*$:
\be \label{smooth:modulation:inter1}
\frac{1}{\nu}\leq 2s_0+s-s_0+O(C(z^*)s^{-\frac 13}(s-s_0))= s+s_0+O(C(z^*)s^{-\frac 13}(s-s_0))\leq 3s.
\ee
One finds similarly using $s_0/2\leq \nu_0^{-1}$ from \eqref{bd:parametersini-beta=0} that $\frac{1}{\nu}\geq s/3$. This and \eqref{smooth:modulation:inter1} imply the second inequality in \eqref{smooth:bd:boostrap improved parameters}. Finally, if $s_1=\infty$ then \eqref{smooth:modulation:inter3} implies $\nu^{-1}=s+O(s^{2/3})$ and the second inequality in \eqref{smooth:bd:boostrap improved parameters2} follows.\\

\noindent \textbf{Step 3}. \emph{Equation for $\lambda$}. Injecting \eqref{smooth:modulation:inter4} in \eqref{smooth:id:modulation} one finds:
\begin{equation*}
    \frac{\lambda_s}{\lambda}+1 = \frac1s (1+O(s^{-\frac13})) (1+O(s^{-\frac13})) = \frac1s + O(s^{-\frac43}).
\end{equation*}
This implies the first inequality in \eqref{smoothmodulationequations}. Since $\lambda = O(se^{-s})$ from \eqref{bd:parameterstrap-critical}, one has
\begin{equation*}
    \frac{d}{ds}(\frac{e^s\lambda}{s}) = O(s^{-\frac43}).
\end{equation*}
We integrate with time the above equation using $\lambda_0 = s_0 e^{-s_0}$ and find
$$
\lambda(s) = se^{-s}(1+\int_{s_0}^sO(s^{-\frac43})ds)=se^{-s}(1+O(s_0^{-\frac13}))
$$
This implies the first inequality in \eqref{smooth:bd:boostrap improved parameters} for $s_0$ large enough. If $s_1=\infty$ then we set $\tilde\lambda_\infty=1+\int_{s_0}^\infty O(s^{-\frac43})ds$ and rewrite the above equality as:
$$
\lambda(s) = se^{-s}(1+\int_{s_0}^\infty O(s^{-\frac43})ds-\int_{s}^\infty O(s^{-\frac43})ds )= se^{-s}(\lambda_\infty+O(s^{-\frac 13}) ).
$$
This is the first inequality in \eqref{smooth:bd:boostrap improved parameters2}.

\end{proof}

\begin{lemma}[Interior Estimate]\label{smooth:lemma:interior}

For any $z^*\geq 1$ and $ \gamma>0$, there exists a large self-similar time $s_0^*$ such that for any $s_0\geq s_0^*$, for any solution which is trapped on $[s_0,s_1]$, we have for $s\in [s_0,s_1]$:
\begin{eqnarray}\label{smooth:bd:int e1}
\mathcal E_1^2(s)  \leq 2\gamma^2 s^{-\frac43}.
\end{eqnarray}

\end{lemma}

\begin{proof}

Recall \eqref{smooth:def:w}. Multiplying \eqref{smooth:id:equationpazep} by $w\varepsilon_z$ and integrating over $[0,z^*]$, one obtains that
\begin{equation} \label{smooth:interior:id:expression2}
\begin{split}
    &\frac{1}{2}\frac{d}{ds} \int_0^{z^*}  w \varepsilon_z ^2 dz  - \int_0^{z^*} ( \frac{\lambda_s}{\lambda} + \frac{\nu_s}{\nu}  + \phi +\varepsilon)   w \varepsilon_z^2 dz 
     + \int_0^{z^*} (\partial_z^{-1} \phi - \frac{\nu_s}{\nu} z + \partial_z^{-1} \varepsilon) \varepsilon_{zz}  w \varepsilon_z dz
     \\
     &- \int_0^{z^*} \phi' \varepsilon  w \varepsilon_z dz + 
     \int_0^{z^*} \partial_z^{-1} \varepsilon   w \varepsilon_z \phi'' dz \\
     = &-2\nu \Big(\int_0^{\frac{1}{\nu}} (\phi + \varepsilon)^2 (z) dz\Big) \int_0^{z^*} z\phi  w \varepsilon_z dz.
\end{split}
\end{equation}
We now compute all terms in \eqref{interior:id:expression2}.\\

\noindent \underline{Potential and transport terms}. Integrating by parts yields
\begin{equation}  \label{smooth:interior:id:expression}
    \begin{split}
        & - \int_0^{z^*} ( \frac{\lambda_s}{\lambda} + \frac{\nu_s}{\nu}  + \phi +\varepsilon)   w \varepsilon_z^2 dz +\int_0^{z^*} (\partial_z^{-1} \phi - \frac{\nu_s}{\nu} z + \partial_z^{-1} \varepsilon) \varepsilon_{zz}  w \varepsilon_z dz
        \\
        = & \Big(\int_0^{z^*} (\phi+\varepsilon)(\tilde{z})d\tilde{z}- \frac{\nu_s}{\nu}z^* \Big) \frac{1}{2}w(z^*) \varepsilon_z^2(z^*) \\
        &+\int_0^{z^*} \left(- \frac{3}{2} \phi-\frac 32 \varepsilon-\frac{\lambda_s}{\lambda}-\frac{1}{2}\frac{\nu_s}{\nu}-\frac{1}{2}\left(\pa_z^{-1}\phi -\frac{\nu_s}{\nu}z+\pa_z^{-1}\varepsilon \right)\frac{w_z}{w}\right) w\varepsilon_z^2dz 
    \end{split}
\end{equation}
For the boundary term, we know that $\|\varepsilon\|_{L^\infty} \leq C(z^*) s^{-2/3}$ from \eqref{smooth:bd:varepsilonLinfty} and thus using \eqref{smoothmodulationequations}:
\begin{equation}\label{condition-critical:s0-3}
      \int_0^{z^*} (\phi+\varepsilon)(\tilde{z})d\tilde{z}- \frac{\nu_s}{\nu}z^* \geq 1-e^{-z^*} - \|\varepsilon\|_{L^\infty} z^*+O(C(z^*)s^{-1}) \geq 1-e^{-z^*} -C(z^*)s^{-\frac 23} \geq 0
\end{equation} 
when $s_0$ is large enough. Since the weight function is $w = z^{-2}$, one has using $\phi=e^{-z}$, \eqref{smoothmodulationequations} and \eqref{smooth:bd:varepsilonLinfty}:
\begin{equation}  \label{smooth:interior:inter1}
\begin{split}
    &\left(- \frac{3}{2} \phi-\frac 32 \varepsilon-\frac{\lambda_s}{\lambda}-\frac{1}{2}\frac{\nu_s}{\nu}-\frac{1}{2}\left(\pa_z^{-1}\phi -\frac{\nu_s}{\nu}z+\pa_z^{-1}\varepsilon \right)\frac{w_z}{w}\right)
    \\
    =&\left(- \frac{3}{2} e^{-z}+O(s^{-\frac 23})+1+O(s^{-1})+O(s^{-1})-\frac{1}{2}\left(1-e^{-z} +O(s^{-1}z)+O(s^{-\frac 23}z) \right)(-\frac{2}{z})\right)
    \\
&   =1-\frac{3}{2}e^{-z}+\frac{1-e^{-z}}{z}+O(s^{-\frac 23}) \geq \frac 12+O(s^{-\frac 23})
\end{split}
\end{equation}
(where we have used the fact that $1-e^{-z} \geq ze^{-z} $. Injecting \eqref{condition-critical:s0-3} and \eqref{smooth:interior:inter1} in \eqref{smooth:interior:id:expression} shows:
\begin{equation}  \label{smooth:interior:bd:transport}
 - \int_0^{z^*} ( \frac{\lambda_s}{\lambda} + \frac{\nu_s}{\nu}  + \phi +\varepsilon)   w \varepsilon_z^2 dz +\int_0^{z^*} (\partial_z^{-1} \phi - \frac{\nu_s}{\nu} z + \partial_z^{-1} \varepsilon) \varepsilon_{zz}  w \varepsilon_z dz     \geq  (\frac 12+O(s^{-\frac 23}))\mathcal E_1(s).
\end{equation}

\noindent \underline{The nonlocal terms}. By direct computations, using $|\varepsilon (z)|\leq \mathcal E_1 \sqrt{\int_0^z w^{-1}}$, Cauchy-Schwarz, and $\phi=e^{-z}$, one gets:
\begin{equation*}
\begin{split}
    \int_0^{z^*} \Big| \phi' \varepsilon  w \varepsilon_z \Big| dz &\leq \Big(\int_0^{z^*} |\phi'|^2 w ( \int_0^{z} \frac{1}{w}(\widetilde{z}) d\widetilde{z}) dz\Big)^{\frac{1}{2}} \mathcal E_1^2 
\leq \Big(\int_0^{\infty} \frac13 ze^{-2z}dz \Big)^{\frac{1}{2}} \mathcal E_1^2 = \frac{1}{2}\sqrt{\frac1{3}} \mathcal E_1^2 ,
\end{split}
\end{equation*}
and 
\begin{equation*}
\begin{split}
    \int_0^{z^*} \Big| \partial_z^{-1} \varepsilon   w \varepsilon_z \phi'' \Big| dz &\leq \Big(\int_0^{z^*} |\phi''|^2 w(\int_0^z (\int_0^{\widetilde{z}} \frac{1}{w}(\xi)d\xi )^{\frac{1}{2}} d\widetilde{z} )^2 dz \Big)^{\frac{1}{2}} \mathcal E_1^2  \leq \Big(\int_0^{\infty} \frac4{75} z^3e^{-2z}dz \Big)^{\frac{1}{2}} \mathcal E_1^2 = \frac{1}{5\sqrt 2} \mathcal E_1^2 . 
\end{split}
\end{equation*}

\noindent \underline{The source term}. Using \eqref{smooth:bd:integral} and Cauchy-Schwarz:
$$
\left| \nu \Big(\int_0^{\frac{1}{\nu}} (\phi + \varepsilon)^2 (z) dz\Big) \int_0^{z^*} z\phi  w \varepsilon_z dz \right| \leq C s^{-1} \mathcal E_1 \sqrt{\int_0^{z^*} z^2\phi^2wdz} \leq C s^{-1}\mathcal E_1\leq  \frac{\mathcal E_1^2}{200}+ C s^{-2}.
$$

\noindent \underline{Conclusion}. Injecting \eqref{smooth:interior:bd:transport} and the three inequalities above in \eqref{smooth:interior:id:expression2} yields:
\begin{equation}
   \frac{d}{ds}\mathcal E_1^2 + \Big(1 - \sqrt{\frac13} - \frac{\sqrt{2}}5 - \frac1{200} - C({z^*})s_0^{-\frac23} \Big)\mathcal E_1^2  \leq \frac{C}{s^2}.  
\end{equation}
Choosing $s_0$ large enough so that $C({ z^*})s^{-\frac32}\leq\frac1{200}$. Then $   1 - \sqrt{\frac13} - \frac{\sqrt{2}}5 - \frac1{200} - C({z^*})s_0^{-\frac23} \geq \frac18,$ and consequently, 
$$
  \frac{d}{ds}\mathcal E_1^2 + \frac18 \mathcal E_1^2 \leq \frac{C}{s^2}.  
$$
When $s_0$ is large enough, we have
\begin{equation}
    \frac{d}{ds}(s^2 \mathcal E_1^2 e^{\frac s{16}}) = e^{\frac s{16}} s^2 \frac{d}{ds}\mathcal E_1^2 + 2s e^{\frac s{16}} \mathcal E_1^2 + \frac{s^2}{16} \mathcal{E}_1^2 e^{\frac s{16}}  \leq e^{\frac s{16}} s^2 \frac{d}{ds}\mathcal E_1^2 + e^{\frac s{16}} \frac{s^2}8 \mathcal E_1^2 \leq C e^{\frac s{16}}.
\end{equation}
Integrating both hand sides from $s_0$ to $s$, since $\mathcal E_1^2(s_0) \leq  \gamma^2 s_0^{-\frac43}$, when $s_0$ is large enough, one obtains
\begin{equation}
   \mathcal E_1^2(s) \leq \Big(s_0^{2}\mathcal E_1^2(s_0) e^{\frac {s_0}{16}} + C e^{\frac s{16}}\Big) e^{-\frac s{16}}\frac1{s^{2}} \leq 2\gamma^2 s^{-\frac43}.
\end{equation}

\end{proof}

The following lemma is similar to Lemma \ref{lemma:main}.

\begin{lemma}[Exterior Estimate]\label{smooth:lemma:main-critical}
There exists $\bar z^*\geq 1$, and for any $z^*\geq \bar z^*$, a $\gamma^*>0$ such that for $0<\gamma\leq \gamma^*$ the following holds true. There exists $s_0^*$ large enough such that if a solution is trapped on $[s_0,s_1]$ with $s_0\geq s_0^*$, for any time $s_0\leq s\leq s_1$ we have

\begin{eqnarray}\label{bd:int e-critical2}
\mathcal E_2^2(s)\leq \frac14 s^{-\frac43},
\end{eqnarray}

\end{lemma}

\begin{proof}

The proof relies on the maximum principle. We rewrite \eqref{smooth:id:equationep} as
\begin{equation}\label{smooth:equation:epsilon-2}
    \begin{split}
        &\varepsilon_s +\mathcal L \varepsilon    = F .
    \end{split}
    \end{equation}
where the transport operator $\mathcal L$(note that it has a nonlinear part) and the source term are:
\begin{align*}
&\mathcal L v=  - \frac{\lambda_s}{\lambda}v - \frac{\nu_s}{\nu} z v_z - 2\phi v + \partial_z^{-1} \phi v_z    - \varepsilon v + \partial_z^{-1} \varepsilon v_z ,\\
&F = -\partial_z^{-1} \varepsilon \phi' +2\nu\Big(\int_0^{\frac{1}{\nu}} (\phi+\varepsilon)^2(z)dz\Big)\Big(-1+(z+1)\phi\Big).
\end{align*}

\noindent \textbf{Step 1}. \emph{A supersolution for $\pa_s+\mathcal L$ on $[z^*,\nu^{-1}]$}. We introduce
$$
f(s,z)= \frac 12 s^{-\frac 23}
$$
and claim that there exists $z^*$ large enough such that for $s_0$ large enough, for all $ s_0\leq s \leq s_1$ and $z\geq z^*$:
\be \label{smooth:exterior:id:supersolution}
(\pa_s +\mathcal L) f\geq \frac{s^{-\frac 23} }{4}.
\ee
To prove \eqref{smooth:exterior:id:supersolution}, we compute using \eqref{smoothmodulationequations} and \eqref{smooth:bd:varepsilonLinfty}:
\be
(\pa_s +\mathcal L) f = \left(-\frac{2}{3s} - \frac{\lambda_s}{\lambda} - 2e^{-z}  - \varepsilon \right)\frac{s^{-\frac 23}}{2}=  \left(O(s^{-1})+1+O(s^{-1}) +O(e^{-z^*}) +O(s^{-\frac 23}) \right)\frac{s^{-\frac 23}}{2}
\ee
which implies \eqref{smooth:exterior:id:supersolution} upon taking $z^*$ large enough and then $s_0^*$ large enough. \\

\noindent \textbf{Step 2}. \emph{Estimate for the source term}. We claim that for $z^*$ large enough and then for $\gamma$ small enough, for all $s_0\leq s \leq s_1$ and  $z\in [z^*,\nu^{-1})$:
\be \label{smooth:exterior:bd:source}
|F(s,z)|\leq \frac{s^{-\frac 23}}{8}.
\ee
We now prove this inequality. We inject the improved bootstrap bound \eqref{smooth:bd:int e1} in the computation \eqref{smooth:bd:varepsiloninterior} and get:
\be \label{smooth:exterior:bd:interiorvarepsilon}
|\varepsilon (z)|\leq C \gamma s^{-\frac23} z^{\frac 32} \quad \mbox{for }z\in [0,z^*].
\ee
Also, $|\varepsilon(z)|\leq s^{-\frac 23} $ for $z\in [z^*,\nu^{-1})$ using \eqref{smooth:bd:etrap}. Therefore, using $\phi(z)=e^{-z}$:
\begin{equation} \label{smooth:exterior:bd:inter1}
\begin{split}
    |\partial_z^{-1} \varepsilon \phi'| &\leq \|\varepsilon\|_{L^\infty} z |\phi'(z)| \leq C\Big(C(z^*) \gamma s^{-\frac 23}+s^{-\frac 23} \Big)z^*e^{-z^*} \leq  \frac{s^{-\frac 23}}{100},
\end{split}
\end{equation}
where we chose $z^*$ large enough and then $\gamma$ small enough. Next, using \eqref{smooth:bd:integral}, for all $z\geq z^*$:
\be \label{smooth:exterior:bd:inter2}
\left| 2\nu\Big(\int_0^{\frac{1}{\nu}} (\phi+\varepsilon)^2(z)dz\Big)\Big(-1+(z+1)\phi\Big) \right| \leq 8s^{-1} (1+(z+1)e^{-z})\leq \frac{C}{s}.
\ee
Combining \eqref{smooth:exterior:bd:inter1} and \eqref{smooth:exterior:bd:inter2} and taking $s_0^*$ large enough shows \eqref{smooth:exterior:bd:source}.\\

\noindent \textbf{Step 3}. \emph{End of the proof}. We introduce
\be \label{smoothexterior:def:fpm}
f^\pm= \pm \left( f-\varepsilon \right).
\ee
Then using \eqref{smooth:equation:epsilon-2}, \eqref{smooth:exterior:id:supersolution} and \eqref{smooth:exterior:bd:source} one obtains that for $s_0\leq s\leq s_1$ and $z\in [z^*,\nu^{-1}]$:
\be \label{smooth:exterior:inter1}
(\pa_s+\mathcal L)f^+= (\pa_s+\mathcal L)f+F\geq \frac{s^{-\frac 23}}{4}-\frac{s^{-\frac 23}}{8}\geq 0 \quad \mbox{and similarly} \quad (\pa_s+\mathcal L)f^-\leq 0.
\ee
Similar to \eqref{particle-moving}, thanks to \eqref{smooth:bd:boostrap improved parameters2}, one has
\begin{equation}
    -\frac{\nu_s}{\nu}z^* + \int_0^{z^*} (\phi+\varepsilon)(\tilde z) d\tilde z \geq \left(\nu - \sup\limits_{0\leq z\leq \frac1{\nu(s)}} |\varepsilon| \right)z^*\geq \left(\frac1s + O(s^{-\frac43}) -\sqrt{C^*}\tilde K e^{-\frac\delta2 s}\right)z^*\geq 0
\end{equation}
provided that $s_0$ is large enough. From this we know that the particles are always moving from region $0\leq z\leq z^*$ to $z^*\leq z \leq \frac1\nu$.
At the boundary $z=z^*$ one has using \eqref{smooth:exterior:bd:interiorvarepsilon} that:
\be \label{smooth:exterior:inter2}
f^+(s,z^*)= \frac{s^{-\frac 23}}{2}-\varepsilon (s,z^*)\geq (\frac{1}{2}-C z^{*\frac 32}\gamma s^{-\frac 23})s^{-\frac 23} \geq 0 \quad \mbox{and similarly} \quad f^-(s,z^*)\leq 0,
\ee
provided $\gamma $ is small enough depending on $z^*$. At initial time $s=s_0$, we have using \eqref{bd:eini-beta=0} that for all $z\in [z^*,\nu_0^{-1}]$:
\be \label{smooth:exterior:inter3}
f^+(s_0,z)\geq \frac{s_0^{-\frac 23}}{2}-\| \varepsilon_0\|_{L^\infty [z^*,\nu_0^{-1}]}\geq \frac{s_0^{-\frac 23}}{2}-\frac{s_0^{-\frac 23}}{4}\geq 0,  \quad \mbox{and similarly} \quad f^-(s_0,z)\leq 0.
\ee
From \eqref{partical-stay}, we know that the particle on the boundary point $z=\frac1\nu$ does not move. This together with \eqref{smooth:exterior:inter3} imply that $f^+(s,\frac1\nu) \geq 0$ and $f^-(s,\frac1\nu)\leq 0$.
Therefore, in view of \eqref{smooth:exterior:inter1}, \eqref{smooth:exterior:inter2} and \eqref{smooth:exterior:inter3} one can apply the maximum principle and obtain that $f^+(s,z)\geq 0$ and $f^-(s,z)\leq 0$ for all $s_0\leq s \leq s_1$ and $z^*\leq z \leq \nu^{-1}$. By the definition \eqref{smoothexterior:def:fpm}  of $f^\pm$ this implies the desired estimate \eqref{bd:int e-critical2} and completes the proof of the Lemma.

\end{proof}

We can now end the proof of Proposition \ref{smooth:pr:bootstrap}.

\begin{proof}[Proof of Proposition \ref{smooth:pr:bootstrap}] Proposition \ref{smooth:pr:bootstrap} is implied by Lemmas \ref{lemma:smoothmodulation}, \ref{smooth:lemma:interior} and \ref{smooth:lemma:main-critical}. The reasoning is similar, and actually simpler since fewer parameters are involved, to the proof of Proposition \ref{pr:bootstrap} which has been done for the case $\beta>0$. Thus, we omit it.

\end{proof}

We can now end the proof of Theorem \ref{theorem:critical}.

\begin{proof}[Proof of Theorem \ref{theorem:critical}] We take $\beta=0$ and $\phi(z)=\phi_{0}(z)=e^{-z}$. Let the constants $z^*,s_0^*\geq 1$ and $\gamma>0$ be given by Proposition \ref{smooth:pr:bootstrap}. Let then, where $\kappa$ is fixed shortly after:
\be \label{smooth:id:defparameters}
\lambda_0\leq \lambda_0^*/2, \quad \mbox{and} \quad \frac{2}{3\log (\lambda_0^{-1})}\leq \nu_0\leq\frac{3}{2\log (\lambda_0^{-1})}
\ee
and an initial datum $a_0$ of the form \eqref{smooth:id:initial} satisfying \eqref{smooth:id:initialcond}. We claim that for $\kappa>0$ small enough, there exist parameters $\bar \lambda_0=\lambda_0(1+O(\lambda_0\kappa))$ and $\bar \nu_0=\nu_0(1+O(\lambda_0\kappa))$, then $a_0$ is trapped in the sense of Definition \ref{def:ini crit-beta=0} with framework parameters $z^*,s_0^*,\gamma$ defined just above and decomposition parameters $\bar \lambda_0,\bar \nu_0$. The proof of this claim is so similar (and simpler since the profile is smooth) to the proof of the analogue claim in the proof of Theorem \ref{nonsmooth} for $\beta >0$, that we omit the details and refer the reader to that proof.\\

Thus, applying Proposition \ref{smooth:pr:bootstrap}, one obtains that the solution $a$ is trapped for all self-similar times $s\in [s_0,\infty)$. We invert the self-similar transformation \eqref{id:self-similarvariables} using \eqref{smooth:bd:boostrap improved parameters2} and define $T=\int_{\bar s_0}^\infty \lambda(s)ds<\infty$ so that:
$$
t(s)=\int_{\bar s_0}^s \lambda(\tilde s)d\tilde s=T-\int_{s}^\infty \tilde \lambda_\infty \tilde s e^{-\tilde s}(1+O(\tilde s^{-\frac 13}))d\tilde s=T-\tilde \lambda_\infty se^{-s}+O(s^{\frac 23}e^{-s})
$$
and hence $\tilde\lambda_\infty se^{-s}=(T-t)+O\left((T-t)|\log (T-t)|^{-1/3}\right)$. We then get using \eqref{smooth:bd:varepsilonLinfty} and \eqref{smooth:bd:boostrap improved parameters2} that $\| \varepsilon \|_{L^\infty(0,\nu^{-1})}\leq C |\log (T-t)|^{-2/3}$, $\lambda=(T-t)+O((T-t)|\log (T-t)|^{-1/3})$ and $\nu=|\log (T-t)|^{-1}+O(|\log (T-t)|^{-4/3})$. Injecting these estimates in the original variables \eqref{decomp} shows the desired estimates \eqref{smooth:1} and \eqref{smooth:2} with $\delta = \frac13$.

\end{proof}

\section*{Acknowledgments}
The work of C. Collot was funded by CY Initiative of Excellence (Grant "Investissements d'Avenir" ANR-16-IDEX-0008).  The work of S. I. was supported by NSERC grant (371637-2019). S. I. would like to thank very much both programs ``Mathematical problems in fluid dynamics" at MSRI spring 2021, and ``Hamiltonian Methods in Dispersive and Wave Evolution Equations" at ICERM fall 2021, for providing him an inspirational environment to work on this project.

\end{document}